\documentclass[11pt,oneside, reqno]{amsart}
\usepackage[top=25mm, bottom=25mm, left=20mm, right=20mm]{geometry}
\usepackage{amsfonts,amssymb,amsmath,amsthm,amsxtra}
\usepackage[T1]{fontenc}
\usepackage[utf8]{inputenc}
\usepackage{bm}
\usepackage{mathtools}
\usepackage{dsfont}
\usepackage{mathrsfs}
\usepackage{enumitem}
\allowdisplaybreaks

\usepackage{graphics}
\usepackage{hyperref}

\numberwithin{equation}{section}
\newtheorem{theorem}[equation]{Theorem}

\newtheorem{lemma}[equation]{Lemma}
\newtheorem{proposition}[equation]{Proposition}
\theoremstyle{definition}
\newtheorem{remark}[equation]{Remark}
\newtheorem{definition}[equation]{Definition}

\DeclareMathOperator\supp{supp}

\author{Leonidas Daskalakis}
\address[Leonidas Daskalakis]{Institute of Mathematics,
Polish Academy of Sciences,
\'Sniadeckich 8,
00-656 Warsaw, Poland \& Institute of Mathematics, University of Wroc\l aw, Plac Grunwaldzki 2/4, 50-384 Wroc\l aw, Poland}
\email{ldaskalakis@impan.pl, leonidas.e.daskalakis@gmail.com}

\author{Mariusz Mirek}
\address[Mariusz Mirek]{
Department of Mathematics, Rutgers University, Piscataway, NJ 08854-8019, USA
\& Institute of Mathematics, University of Wroc\l aw, Plac Grunwaldzki 2/4, 50-384 Wroc\l aw, Poland}
\email{mariusz.mirek@rutgers.edu}

\author{M\'at\'e Wierdl}
\address[M\'at\'e Wierdl]{Department of Mathematics, University of Memphis, Memphis, TN 38152, USA}
\email{mwierdl@memphis.edu}

\thanks{Mariusz Mirek was partially supported by   NSF CAREER grant DMS-2236493.}

\begin{document}

\title[Pointwise ergodic theorems along sequences 
of intermediate growth]{Pointwise ergodic theorems along sequences \\
of intermediate growth}
\maketitle
\begin{abstract}
We establish the first pointwise convergence result for ergodic
averages with iterates along explicit and deterministic sequences of
intermediate growth, that is, growing faster than any polynomial but
slower than any exponential. In particular, we show that the sequence
$(\lfloor \exp((\log n)^c)\rfloor)_{n\in\mathbb{Z}+}$, with
$c\in(1,8/7)$, is universally $L^p$-good for every $p\in(1,\infty]$. This gives an
affirmative answer to an open problem dating back to the mid 1980s and
contributes to Bellow's program, initiated in the earlier part of the same decade, on the
characterization of $L^p$-good sequences in pointwise ergodic
theorems.

The proof combines the so-called one-frequency circle method with a delicate application of Vinogradov's method for estimating exponential sums whose phases involve $\big(\lfloor \exp((\log n)^c)\rfloor\big)_{n\in\mathbb{Z}_+}$. An interesting feature of our analysis, reminiscent of estimates arising in the study of the zero-free region of the Riemann zeta function, is that the argument relies on the classical Vinogradov method, in the sense that it necessitates estimates on the number of solutions for the Vinogradov system of Diophantine equations with explicit dependence on the system's parameters.
\end{abstract}
\section{Introduction}

\subsection{Historical background and statement of the main results}
A classical theorem of Blum and Hanson~\cite{BH}, dating back to 1960, asserts that if $T$ is an invertible strongly mixing measure-preserving transformation on a probability space $(X,\mathcal B(X),\mu)$, then for every strictly increasing sequence  of integers ${\bm a}=(a_n)_{n\in\mathbb Z_+}$ and every $f\in L^1(X)$ the averages
\begin{align}
\label{eq:26}
A_{N;T}^{{\bm a}}f(x)\coloneqq \frac{1}{N}\sum_{n=1}^N f(T^{a_n}x), \qquad x\in X,
\end{align}
converge in $L^1(X)$ norm to $\int_X f\,d\mu$; conversely, this property
characterizes strong mixing. The situation for pointwise convergence
is fundamentally different. In 1971, Krengel~\cite{Krengel} was the
first to show that there exists a strictly increasing sequence
 of integers ${\bm a}=(a_n)_{n\in\mathbb Z_+}$ such that for every
aperiodic measure-preserving transformation one can find a measurable set $A$ for which
$A_{N;T}^{{\bm a}}\mathds{1}_A$ diverges almost everywhere. About a decade later, Bellow~\cite{Bellow} proved that if ${\bm a}=(a_n)_{n\in\mathbb Z_+}$ is a lacunary sequence of integers, i.e. $\inf_{n\in\mathbb Z_+}\frac{a_{n+1}}{a_n}>1$, then for every aperiodic measure-preserving system and every $p\in[1,\infty)$ there exists a function $f\in L^p(X)$ such that the  averages $A_{N;T}^{{\bm a}}f$ diverge almost everywhere.

These results naturally led Bellow to initiate a program aimed at
identifying the sequences along which the pointwise ergodic theorem
holds universally. More precisely, at Oberwolfach in 1981,
Bellow~\cite{Bel} asked for interesting classes of sequences ${\bm a}=(a_n)_{n\in\mathbb Z_+}$ for which the averages
$A_{N;T}^{{\bm a}}f$ converge almost everywhere for every
$f\in L^p(X)$ with $p\in[1, \infty]$ and, more ambitiously, for an
intrinsic characterization of such sequences. Among the basic examples
Bellow singled out were the sequences of squares and primes. This
question was also independently formulated by Furstenberg~\cite{Fur3}.

Following the terminology that emerged from Bellow's program, we introduce the following definition.

\begin{definition}[Universally $L^p$-good/bad sequences]
\label{def:1}
Let ${\bm a}=(a_n)_{n\in\mathbb Z_+}\subseteq\mathbb Z$ and $p\in[1, \infty]$ be given. 
\begin{enumerate}[label*={\arabic*.},  itemsep=5pt]
\item We say that ${\bm a}=(a_n)_{n\in\mathbb Z_+}$ is
\emph{universally $L^p$-good} if for every probability space
$(X,\mathcal B(X),\mu)$, every invertible measure-preserving
transformation $T:X\to X$, and every $f\in L^p(X)$, the averages
$A_{N;T}^{{\bm a}}f(x)$ converge for $\mu$-almost every
$x\in X$.

\item We say that ${\bm a}=(a_n)_{n\in\mathbb Z_+}$ is
\emph{universally $L^p$-bad} if for every probability space
$(X,\mathcal B(X),\mu)$, and every aperiodic invertible
measure-preserving transformation $T:X\to X$, there exists a function
$f\in L^p(X)$ such that the averages $A_{N;T}^{{\bm a}}f(x)$ diverge
on a set of positive measure in $X$.
\end{enumerate}
\end{definition}
In view of this definition, we immediately see that the sequence of integers $(n)_{n\in\mathbb Z_+}$ is universally $L^p$-good for every $p\in[1,\infty]$ by the Birkhoff pointwise ergodic theorem~\cite{Bir}. At the other extreme, it follows from Bellow's work~\cite{Bellow} that every lacunary sequence of integers is universally $L^p$-bad for every $p\in[1,\infty)$.
It was also shown in~\cite{Bellow2} that lacunary sequences satisfy the so-called strong sweeping out property, which, in particular, implies that pointwise convergence fails even for bounded functions.

Using the terminology from Definition~\ref{def:1}, we can now state the main result of the present work.

\begin{theorem}
\label{thm:main}
If $c\in(1, 8/7)$, then  $\big(\big\lfloor\exp\big((\log n)^c\big)\big\rfloor\big)_{n\in\mathbb Z_+}$ is universally $L^p$-good for every $p\in(1, \infty]$.
\end{theorem}

To the best of the authors' knowledge, this is the first explicit and deterministic
example of a sequence of intermediate growth, that is, growing faster
than any polynomial but slower than any exponential, that is
universally $L^p$-good for every
$p\in(1,\infty]$. Theorem~\ref{thm:main} gives an affirmative answer
to an open problem dating back to the mid 1980s. Before discussing how the aforementioned theorem advances Bellow's program, let us briefly
review the current state of the art. Bellow's question~\cite{Bel} on
the characterization of $L^p$-goodness and $L^p$-badness initiated a
systematic and successful study of a wide variety of sequences,
leading to a number of fundamental examples of universally $L^p$-good
and $L^p$-bad sequences:
\begin{enumerate}[label*={\arabic*.},  itemsep=3pt]

\item If $P\colon\mathbb Z\to\mathbb Z$ is a polynomial with integer coefficients, then the sequence $(P(n))_{n\in\mathbb Z_+}$ is universally $L^p$-good for every $p\in(1,\infty]$. This follows from a series of groundbreaking papers by Bourgain~\cite{B1,B2,B3} from the mid-to-late 1980s, which introduced powerful tools from harmonic analysis, analytic number theory, and probability theory into the study of universal $L^p$-goodness.

\item The sequence of prime numbers is universally $L^p$-good for every $p\in(1,\infty]$, as was established by Bourgain~\cite{B3} and the third author~\cite{Wierdl}. Sequences of polynomial values over the primes were proved to be universally $L^p$-good for every $p\in(1,\infty]$ by Nair~\cite{Nair}; see also \cite{Trojan, Mes}.

\item If $P\colon\mathbb R\to\mathbb R$ is a polynomial with real coefficients, then the sequence $(\lfloor P(n)\rfloor)_{n\in\mathbb Z_+}$ is universally $L^p$-good for every $p\in(1,\infty]$, as was also established by Bourgain in~\cite{B3}. Bourgain's papers~\cite{B1,B2,B3} not only answered Bellow's question in the affirmative but also secured the role of the Calder{\'o}n transference principle~\cite{C1} as a bridge between pointwise ergodic theorems and the study of discrete analogues in harmonic analysis, thereby bringing a new perspective to Bellow's program.

\item The sequence $(\lfloor n^c\rfloor)_{n\in\mathbb Z_+}$, with
non-integer $c>1$, is universally $L^p$-good for every
$p\in(1,\infty]$, as was shown by the third author in his PhD
thesis~\cite{Wierdlphd} in the late 1980s. For
$c\in(0,1)$, the sequence
$(\lfloor n^c\rfloor)_{n\in\mathbb Z_+}$ is universally $L^1$-good,
whereas $(\lfloor (\log n)^c\rfloor)_{n\in\mathbb Z_+}$ is universally
bad even for bounded functions for every $c>0$, see \cite[Example~2.18]{JW}.

\item It is striking that arithmetic properties can play a decisive role here. On the one
hand, the third author~\cite{Wierdlphd} showed that the sequence
$(\lfloor n\log n\rfloor)_{n\in\mathbb Z_+}$ is universally $L^p$-good
for every $p\in(1,\infty]$, while the sequence
$(n\lfloor\log n\rfloor)_{n\in\mathbb Z_+}$ is universally bad even
for bounded functions. On the other hand, the first author \cite{brpoly} recently
showed that the sequence
$(n\lfloor n\sqrt{q}\rfloor)_{n\in\mathbb Z_+}$ is universally
$L^2$-good for every
$q\in\mathbb Q_{+}$, leaving open the question  for $p<2$.

\item In the context of such questions, the third author's PhD thesis~\cite{Wierdlphd} initiated the systematic study of sequences of the form $(\lfloor h(n)\rfloor)_{n\in\mathbb Z_+}$, where $h$ belongs to a Hardy field, that is, a field of germs at infinity of real-valued functions that is closed under differentiation. This line of research led to two works of the same author with Boshernitzan~\cite{Bowie} and with Boshernitzan, Kolesnik and Quas~\cite{BKQW}, which studied Hardy field functions of polynomial growth and provided a variety of sufficient conditions for $L^p$-goodness and $L^p$-badness in terms of their growth at infinity and their distance from polynomials. An important message conveyed by these papers is that $L^p$-goodness and $L^p$-badness are delicate phenomena that may vary considerably from one sequence to another. This was illustrated by two examples:
\smallskip
\begin{itemize}[itemsep=3pt]
\item[(i)] The sequence $(\lfloor n^k+\log n\rfloor)_{n\in\mathbb Z_+}$ is universally $L^2$-bad for any $k\in\mathbb Z_+$.

\item[(ii)]
The sequence $(\lfloor \sqrt{2}n^2+n+\log n\rfloor)_{n\in\mathbb Z_+}$ is universally $L^p$-good for every $p\in(1,\infty]$, while, for some $\theta\in\mathbb R\setminus\mathbb Q$, the sequence $(\lfloor \theta n^2+n+\log n\rfloor)_{n\in\mathbb Z_+}$ may fail to be universally $L^p$-good. It remains open, for example, whether the sequence $(\lfloor \pi n^2+n+\log n\rfloor)_{n\in\mathbb Z_+}$ is universally $L^p$-good for every $p\in(1,\infty]$, which is rather interesting.
\end{itemize}

\item An important class of sequences studied in~\cite{BKQW, Bowie} that are universally $L^p$-good for every $p\in(1,\infty]$ consists of sequences of the form $(\lfloor n^cL(n)\rfloor)_{n\in\mathbb Z_+}$ for a broad class of slowly varying functions $L$. In particular, one may take
\begin{align}
\label{eq:27}
L(x)=(\log x)^A,\qquad
L(x)=e^{A(\log x)^B},\qquad
L(x)=\underbrace{\log\circ\cdots\circ\log}_{m\text{ times}}x,
\end{align}
where $A\in\mathbb R$ when $c>1$, while $A\in\mathbb R_+$ when $c=1$, and $B\in(0,1)$ and $m\in\mathbb Z_+$ in both cases.
\end{enumerate}

Bellow's program has flourished for more than four decades, connecting
different fields and revealing new phenomena, and the examples above
illustrate only a small part of the rich and extensive literature that emerged from this program. For a comprehensive discussion on
the subject and the related literature, we refer the reader to~\cite{BKQW, Bowie,RW}. In view of the aforementioned results, in 1991 Rosenblatt together with the third author conjectured  that there
are no sequences ${\bm a}=(a_n)_{n\in\mathbb Z_+}\subseteq\mathbb Z$
with gaps tending to infinity, that is,
$\lim_{n\to\infty}(a_{n+1}-a_n)=\infty$, that are universally
$L^1$-good, see \cite[Conjecture~4.1, p.~71]{RW}. This was disproven in the early 2000s by 
Urban and Zienkiewicz~\cite{UZ} by showing that the sequence
$(\lfloor n^c\rfloor)_{n\in\mathbb Z_+}$ is universally $L^1$-good for
$1<c<1.001$.

In the following years, several sparse sequences were shown to be $L^1$-good
and $L^1$-bad. More precisely:
\begin{enumerate}[label*={\arabic*.},  itemsep=3pt]

\item The second author~\cite{Mirek} showed that the sequences
$(\lfloor n^cL(n)\rfloor)_{n\in\mathbb Z_+}$, with slowly varying functions $L$ as
in~\eqref{eq:27}, are universally
$L^1$-good for $1<c<30/29$, providing further counterexamples to the
Rosenblatt--Wierdl conjecture. As far as we know, the case $c=1$
remains open. The first author~\cite{WT11LD} extended the aforementioned result for a wide class of thin subsets of the integers exhibiting similar
arithmetic features.

\item The sequence $(n^k)_{n\in\mathbb Z_+}$ is universally $L^1$-bad
for every integer $k\ge2$, the case $k=2$ is due to Buczolich and
Mauldin~\cite{BM}, while the general case is due to
LaVictoire~\cite{LaV1}. It remains open whether $(P(n))_{n\in\mathbb Z_+}$ is universally $L^1$-bad for every nonlinear polynomial $P$ with integer coefficients. On the other hand, a result of Christ~\cite{C} asserts that for every fixed $k\in\mathbb Z_+$, there exists a universally $L^1$-good sequence ${\bm a}=(a_n)_{n\in\mathbb Z_+}$ with $\frac{a_n}{n^k}$ bounded above and below by positive constants. In other words, there exist universally $L^1$-good sequences of every polynomial growth, see also Remark~3.1 in the aforementioned work.

\item Let us also note that the sequence of prime numbers is also universally $L^1$-bad,
see \cite{LaV1}.
\end{enumerate}

To the best of our knowledge, every known example  of a deterministic universally $L^p$-good sequence ${\bm a}=(a_n)_{n\in\mathbb Z_+}$ prior to the present work is of at most polynomial growth; that is, there exists $k\in\mathbb Z_+$ such that
\[
\lim_{n\to\infty}\frac{a_n}{n^k}=0.
\]

In contrast, $L^p$-badness extends well into the regime of intermediate growth. Recently, Mondal, Roy, and the third author~\cite{MW}, extending earlier results from \cite{JW}, proved that sequences ${\bm a}=(a_n)_{n\in\mathbb Z_+}$ satisfying the sublacunarity condition
\begin{equation}\label{growthcond}
\frac{a_{n+1}}{a_n}\ge 1+\frac{1}{(\log\log n)^{\theta}}
\quad\text{for some} \quad \theta\in(0,1),
\end{equation}
are strong sweeping out, and in particular, universally $L^p$-bad for every $p\in[1,\infty]$.

Clearly, there is a vast gap between the polynomial growth of all known examples of universally $L^p$-good sequences and the obstructions of the form \eqref{growthcond} appearing in the literature. The first attempts to close this gap and obtain an $L^p$-goodness result beyond the polynomial regime go back to the third author's PhD thesis~\cite{Wierdlphd}, where the sequence ${\bm a}_c\coloneqq \big(\big\lfloor\exp\big((\log n)^c\big)\big\rfloor\big)_{n\in\mathbb Z_+}$, for $c\in(1,3/2)$, was studied. Karatsuba's earlier work~\cite{Kar} on the equidistribution of the fractional parts associated with such sequences makes this a natural candidate to consider. In the context of establishing pointwise convergence of the corresponding ergodic averages, however, Karatsuba's estimates alone proved insufficient. Nevertheless, using the available estimates, the third author was able to show that for every $c\in(1,3/2)$ one can construct a strictly increasing sequence $(N_k)_{k\in\mathbb Z_+}\subseteq\mathbb Z_+$ such that $A_{N_k;T}^{{\bm a}_c}f$ converges pointwise almost everywhere on $X$ for every $f\in L^p(X)$ with $p\in(1,\infty]$.

Our considerations leading to Theorem~\ref{thm:main} begin from Karatsuba's exponential-sum estimates~\cite{Kar} and, at the same time, demonstrate that new input in estimating the corresponding exponential sums is required to establish pointwise convergence for such sequences; see also the discussion following Proposition~\ref{expsumestfull}.

We obtain Theorem~\ref{thm:main} as a corollary of the following quantitative pointwise ergodic theorem.

\begin{theorem}
\label{thm:mainq}
Fix $c\in(1,8/7)$ and let $h(x)\coloneqq h_c(x)\coloneqq\exp\big((\log x)^c\big)$. Assume that $(X,\mathcal{B}(X),\mu)$ is a $\sigma$-finite measure space endowed with an invertible measure-preserving transformation $T\colon X\to X$. Then, for every $p\in(1,\infty)$ and every $f\in L^p(X)$, the averages
\begin{align}
\label{eq:28}
A_{N;T}^{\lfloor h\rfloor}f(x)
\coloneqq \frac{1}{N}\sum_{n=1}^N
f\big(T^{\lfloor e^{(\log n)^c}\rfloor}x\big),
\quad x\in X,
\end{align}
converge pointwise $\mu$-almost everywhere on $X$ as $N\to\infty$. Moreover, for all $p\in(1,\infty)$ and $r\in(2,\infty)$, there exists a  constant $C=C(p,c)\in\mathbb R_+$ such that for every $f\in L^p(X)$ we have
\begin{align}
\label{eq:29}
\big\|V^r(A_{N;T}^{\lfloor h\rfloor}f:N\in\mathbb{Z}_+)\big\|_{L^p(X)}
\le \frac{Cr}{r-2}\|f\|_{L^p(X)},
\end{align}
where $V^r$ denotes the $r$-variational seminorm defined in \eqref{eq:18}.
\end{theorem}

A few comments  are now in order.

\begin{enumerate}[label*={\arabic*.},  itemsep=3pt]

\item It is clear that the $r$-variational inequality in \eqref{eq:29} implies that the sequence $\big(A_{N;T}^{\lfloor h\rfloor}f\big)_{N\in\mathbb Z_+}$ is Cauchy almost everywhere on $X$ and consequently proves Theorem~\ref{thm:main}.

\item In ergodic applications, finite measure spaces are of primary importance. However, by including $\sigma$-finite measure spaces in Theorem~\ref{thm:mainq}, the integer shift system arising in the Calder{\'o}n transference principle becomes available. This is the  reason why Theorem~\ref{thm:mainq} is formulated in this generality.

\item Inequality  \eqref{eq:29} yields the corresponding maximal estimate, namely, it implies that for all $p\in(1,\infty]$, there exists a  constant $C=C(p,c)\in\mathbb R_+$ such that, for every $f\in L^p(X)$, we have
\begin{align}
\label{eq:30}
\big\|\sup_{N\in\mathbb{Z}_+}\big|A_{N;T}^{\lfloor h\rfloor}f\big|\big\|_{L^p(X)}
\le C\|f\|_{L^p(X)}.
\end{align}

\item The range of the parameter $r\in(2,\infty)$ in inequality \eqref{eq:29} is sharp. With straightforward modifications, the proof also yields uniform $2$-oscillation and uniform $2$-jump estimates on $L^p(X)$ for all $p\in(1,\infty)$. We refer to \cite{MOE} for the relevant definitions of  $2$-oscillations and  $2$-jumps. Moreover, with no additional difficulty and by adapting the strategy of \cite{frpr}, see sections 1 and 2 therein, one may also establish multiparameter oscillation inequalities for the corresponding ergodic averages. 

\item The range of the parameter $p\in(1,\infty)$ in inequality \eqref{eq:29} is also sharp. However, a natural question that arises is whether the weak type $(1,1)$ estimate holds. To extend Theorem~\ref{thm:main} to $p=1$, it would suffice to establish a weak type $(1,1)$ estimate for the corresponding maximal function, as in \eqref{eq:30}.
This is an interesting question that would also  contribute an ``extreme'' counterexample of the Rosenblatt--Wierdl conjecture, going beyond \cite{C}, where the existence of arbitrarily polynomially sparse universally $L^1$-good sequences is established, and further showcase that sparseness and $L^1$-goodness are completely different phenomena. We plan to address this question in the near future.

\item We believe that the present approach can be adapted to accommodate more general sequences of intermediate growth. For instance, with minor modifications, it is well within the reach of the method to address orbits of the form
\[
\big(\big\lfloor n^k \exp\big((\log n)^c\big)+R(n)\big\rfloor\big)_{n\in\mathbb{Z}_+},
\]
where $k\in\mathbb{Z}_+$, $c\in(1,8/7)$, and  $R(x)$ is a function of one of the forms
\[
Ax^a\underbrace{\log\circ\ldots\circ\log x}_{m\text{ times}},
\qquad
Ax^a(\log x)^B,
\qquad
Ax^a e^{B(\log x)^b},
\]
with $a,A,B\in\mathbb{R}$, $m\in\mathbb{Z}_+$, and $b\in(0,1)$. However, for clarity of presentation, we restrict our attention to the sequence $\big(\big\lfloor\exp\big((\log n)^c\big)\big\rfloor\big)_{n\in\mathbb Z_+}$ and leave the details to the interested reader.

\item Extending Theorem~\ref{thm:mainq} to functions $h$ of the form $\exp(\rho(\log x))$, where $\rho(x)=x^cL(x)$ and $L$ is a slowly varying function such as those appearing in \eqref{eq:27}, would require establishing suitable derivative estimates of arbitrary order for such functions. This task would necessitate nontrivial extensions of \eqref{lowerboundder} and \eqref{upperboundder}; see also Section~8 of \cite{ApKar}. This remains open.
\end{enumerate}

\subsection{Proof strategy for Theorem~\ref{thm:mainq}}\label{Strategy}
In view of the remarks following Theorem~\ref{thm:mainq}, it suffices to prove inequality \eqref{eq:29}. By the Calder{\'o}n transference principle it suffices to establish the estimate for the integer shift system, that is, $(\mathbb{Z},\mathcal{B}(\mathbb{Z}),\mu_{\mathbb Z})$ equipped with the shift transformation $S\colon \mathbb{Z}\to \mathbb{Z}$, where $\mathcal{B}(\mathbb{Z})$ denotes the $\sigma$-algebra of all subsets of $\mathbb Z$, $\mu_{\mathbb Z}$ denotes the counting measure on $\mathbb Z$, and $S(x)\coloneqq x-1$ for every $x\in\mathbb Z$. Then the averages $A^{\lfloor h\rfloor}_{N;T}f$ from \eqref{eq:26} with the integer shift $T=S$ become
\begin{align}
\label{eq:32}
A^{\lfloor h\rfloor}_{N;\mathbb Z}f(x)
\coloneqq \frac{1}{N}\sum_{n=1}^N f(x-\lfloor h(n)\rfloor),
\qquad x\in\mathbb Z.
\end{align}
Here, we abuse notation and write $A^{\lfloor h\rfloor}_{N;\mathbb Z}$ instead of $A^{\lfloor h\rfloor}_{N;S}$ to emphasize that we are dealing with operators on the integers. It then suffices to prove that for all $p\in(1,\infty)$ and $r\in(2,\infty)$, there exists a constant $C=C(p,c)\in\mathbb R_+$ such that for every $f\in\ell^p(\mathbb Z)$ we have
\begin{align}
\label{eq:33}
\big\|V^r(A^{\lfloor h\rfloor}_{N;\mathbb Z}f:N\in\mathbb{Z}_+)\big\|_{\ell^p(\mathbb Z)}
\le \frac{Cr}{r-2}\|f\|_{\ell^p(\mathbb Z)}.
\end{align}

Once \eqref{eq:33} is established, the Calder{\'o}n transference
principle \cite{C1, Kosz} guarantees that \eqref{eq:29} also holds for
an arbitrary measure-preserving system. We emphasize that the
Calder{\'o}n transference principle does not transfer pointwise almost
everywhere convergence itself. It only transfers quantitative bounds,
such as the inequality in \eqref{eq:33}, from the integer shift system
to the corresponding ones in a $\sigma$-finite measure-preserving
system, as in \eqref{eq:29}, where they are then used to deduce
pointwise almost everywhere convergence. Hence, from now on, we may focus
on establishing inequality \eqref{eq:33}.

There are both advantages and disadvantages to this reduction, the latter including for example completely losing information about the original
measure-preserving system, and having to work with a space of infinite measure. The advantage is clear, we pass from an abstract
measure-preserving system to the integer shift system, the averaging operators \eqref{eq:32} are convolution operators and tools from harmonic analysis become available.

Applying the Fourier transform to the averages in \eqref{eq:32}, see Section~\ref{notation} for definitions, we obtain
\[
\mathcal{F}_{\mathbb{Z}}[A^{\lfloor h\rfloor}_{N;\mathbb Z}f](\xi)
=m^{\lfloor h\rfloor}_{N}(\xi)\cdot\mathcal{F}_{\mathbb{Z}}[f](\xi),
\qquad \xi\in\mathbb{T}\text{,}
\]
where the corresponding multiplier is the normalized exponential sum
\begin{equation}\label{genmultiplier}
m^{\lfloor h\rfloor}_{N}(\xi)\coloneqq\frac{1}{N}\sum_{n=1}^N e(\lfloor h(n)\rfloor\xi),
\end{equation}
and where we have used the notation $e(x)\coloneqq e^{2\pi i x}$ for any $x\in\mathbb{R}$. This simple
observation allows tools from harmonic analysis and number theory to
be employed. The key mechanism for understanding the multipliers \eqref{genmultiplier} is provided by the following proposition.
\begin{proposition}\label{expsumestfullforav}
For every $c\in(1,8/7)$ there exist  $\chi=\chi(c)\in\mathbb R_+$ and $C=C(c)\in\mathbb R_+$ such that for every $N\in\mathbb{Z}_+$ and $\xi\in\mathbb{T}$ we have
\begin{equation}\label{expsumconcrete}
\Big|\sum_{n=1}^Ne(\lfloor h(n)\rfloor\xi)-\sum_{ 2\le m\le h(N)}\varphi'(m)e(m\xi)\Big|\le CNe^{-(\log N)^{\chi}}\text{,}
\end{equation}
where $h(x)=\exp\big((\log x)^c\big)$ and $\varphi=h^{-1}$ is its compositional inverse.
\end{proposition}

Proposition~\ref{expsumestfullforav} determines the proof strategy for inequality \eqref{eq:33}. Standard considerations, essentially relying on simple square function arguments, interpolation and \eqref{expsumconcrete}, reduce the $r$-variational estimates for $A^{\lfloor h\rfloor}_{N;\mathbb{Z}}$ to the analogous ones for the following much more well-behaving averaging operators
\begin{align}
\label{eq:34}
A^{\rm n}_{N;\varphi',\mathbb Z}f(x)
\coloneqq\frac{1}{N}\sum_{2\le m\le h(N)}\varphi'(m)f(x-m)\text{.}
\end{align}
In fact, the above averages  are smoothly-weighted discrete Hardy--Littlewood averaging operators, and their $r$-variational estimates follow from the corresponding ones for the classical discrete Hardy--Littlewood ones, see Section~\ref{sectionOrbit}. Such maneuvers are fairly standard, and the intermediate growth of the sequence does not substantially complicate things here, so let us focus on the proof of Proposition~$\ref{expsumestfullforav}$.

The exponential sum estimate \eqref{expsumconcrete} requires employing the so-called one-frequency circle method, in which the major arc consists of a single narrow arc centered at the origin, while its complement forms the minor arc. This contrasts sharply with the classical circle method, which addresses genuinely  multifrequency phenomena, and where the major arcs consist of many narrow arcs centered at rationals with relatively small denominators in terms of the underlying scale $N$. The one-frequency circle method is typical for exponential sums with floor-non-polynomial phase functions and makes the major arc analysis straightforward.

For the major arc regime, elementary estimates for the difference on the left-hand side of \eqref{expsumconcrete}, explained in Section~\ref{CircleMethod} (see also Remark~\ref{rem:1}), yield an error term of order $O(|\xi|h(N)+1)$ whenever $\xi\in[-1/2,1/2)$. Thus, any restriction of the form $|\xi|\le h(N)^{-1}N^\rho$, where $\rho\in(0,1)$, yields the desired decay in \eqref{expsumconcrete}. However, the region $|\xi|\le h(N)^{-1}N^\rho$ becomes extremely narrow since $h(N)$ has intermediate growth and, in particular, grows faster than any polynomial. Nevertheless, for the argument to work, the major arc must be defined essentially in this way. The analysis is then reduced to studying each exponential sum in \eqref{expsumconcrete} separately over the exceptionally wide complementary minor arc.

The second exponential sum in \eqref{expsumconcrete} satisfies the desired bounds for $|\xi|\ge h(N)^{-1}N^\rho$. This follows by estimating the appropriate geometric sum after removing the weight $\varphi'$ by summation by parts. Obtaining the desired bounds for the first exponential sum in \eqref{expsumconcrete} is considerably more challenging. A key intermediate result which may be of independent interest is the following exponential sum estimate.
\begin{proposition}
\label{expsumestfull}
Let $c\in(1,8/7)$, $\rho\in(0,\infty)$ and $\theta\in(0,c-1)$, and  let $h(x)\coloneqq\exp\big((\log x)^c\big)$. Then there exist  $\chi=\chi(c)\in\mathbb R_+$ and $C=C(c,\rho,\theta)\in\mathbb R_+$ such that for every
$N\in\mathbb{Z}_+$ and $\xi\in\mathbb{R}$ with
$h(N)^{-1}N^{\rho}\le|\xi|\le h(N)^{\theta}$ we have
\begin{equation}\label{expsumestbasicdificult}
\Big|\sum_{N<n\le 2N}e\big(h(n)\xi\big)\Big|\le CNe^{-(\log N)^{\chi}}\text{.}
\end{equation}
\end{proposition}

Estimates of the exponential sums in \eqref{expsumestbasicdificult} go back to the work of Karatsuba~\cite[Theorem~2]{Kar}, where, among other things, the sequence of fractional parts $\big({\exp((\log n)^c)}\big)_{n\in\mathbb Z_+}$ was shown to be equidistributed for every $c\in(1,3/2)$. In fact, using Vinogradov's method for estimating exponential sums, Karatsuba proved a general theorem that yielded the aforementioned equidistribution result. Later, Br{\"u}dern and Perelli~\cite[Theorem~2]{ApKar} used Karatsuba's general result to show that \eqref{expsumestbasicdificult} holds in the frequency region $h(N)^{-(1-\varepsilon)}\le |\xi|\le h(N)^\theta$ for every $c\in(1,3/2)$, $\theta\in(0,c-1)$, and $\varepsilon\in(0,1)$. This refinement, as well as Karatsuba's argument itself, is insufficient for our purposes. More precisely, the lower bound of this frequency range is too large, and Karatsuba's general result cannot, in fact, handle any threshold of the form $|\xi|\ge h(N)^{-1}N^{\kappa}$, where $\kappa\in(0,\infty)$; see Remark~\ref{rem:1} in Section~\ref{CircleMethod}.

To handle this low-frequency range $h(N)^{-1}N^{\rho}\le|\xi|\le h(N)^{-(1-\varepsilon)}$, we apply Vinogradov's method directly and derive inequality \eqref{expsumestbasicdificult} by optimizing the shift parameter, the order of the Taylor approximation, and the precise choice of the set of derivatives whose upper and lower bounds are exploited by the method; see Sections~\ref{VinMethod} and~\ref{sec:exp} for details. Vinogradov's method for estimating exponential sums is inherently technical, and so is its specific manifestation here. We decouple the standard steps of the method from our novel input and, after briefly presenting the basic mechanism of Vinogradov's method in a high degree of generality in Section~\ref{VinMethod}, elaborate on the aforementioned choices; see the discussion after \eqref{eq:4}.

Let us make a few final comments. Our approach deviates from Karatsuba's~\cite{Kar} and, for such small frequencies, we do not obtain the full range $c\in(1,3/2)$, but rather $c\in(1,8/7)$. A key mechanism behind Vinogradov's method for estimating exponential sums is the Vinogradov mean value theorem \cite{V}, which asserts that for every $k,m,n\in\mathbb Z_+$ such that $k\ge n(n+m)$, and for every $P\ge n^{n(1-1/n)^{-m}}$, we have
\begin{align}
\label{eq:35}
J_{k,n}(P)\le 2^{4km}P^{2k-\frac{n(n+1)}{2}+\frac{n(n+1)}{2}(1-1/n)^m},
\end{align}
where
\[
J_{k,n}(P)
\coloneqq\big|\big\{(u_1,\dotsc,u_k,v_1,\dotsc,v_k)\in[P]^{2k}:
\sum_{i\in[k]}(u_i^s-v_i^s)=0
\text{ for every }s\in[n]\big\}\big|.
\]
In the context of our considerations, it is essential that estimate \eqref{eq:35} provides explicit dependence on both parameters $(k,n)$ of the Vinogradov system of Diophantine equations. The superpolynomial growth of the function $h(x)\coloneqq\exp\big((\log x)^c\big)$ forces us to perform a Taylor approximation for $x\in(N,2N]$ of order $n=n(N)\to\infty$ as $N\to\infty$. This, in turn, forces $k$ to depend on $N$ as well and to diverge as $N\to\infty$, making the need for explicit bounds in \eqref{eq:35} clear.

Interestingly, our main result provides an example outside number theory where, much like in estimates of exponential sums arising in the study of zero-free regions for the Riemann zeta function, the proof relies in an essential way on the classical form of the Vinogradov mean value theorem \cite{V}, since any bound for $J_{k,n}(P)$ not providing explicit dependence on the system's parameters is unsuitable. Thus, the breakthrough results of Wooley~\cite{W} and Bourgain--Demeter--Guth~\cite{BDG}, see also \cite{Lillian}, are not applicable in our setting, as they do not provide the required explicit dependence on these parameters. To the best of our knowledge, it remains open whether such explicit dependence can be recovered from their methods, and
it would be interesting to settle this question.

Having at our disposal Proposition~\ref{expsumestfull}, estimating the first exponential sum on the right-hand side of \eqref{expsumconcrete} becomes relatively straightforward. We first establish a nondyadic variant of Proposition~\ref{expsumestfull}, see Proposition~\ref{expsumestfullfullN}, which we then combine with certain standard floor-removing techniques from number theory to conclude; see Subsection~\ref{minarcsubsection}. Although both steps are usually fairly routine, the superpolynomial growth of the function, the large size of the minor arcs, and the subpolynomial nature of the saving in Proposition~\ref{expsumestfull} require some care in the analysis.

\subsection{Organization of the paper}
We have organized the paper as follows. In Section~\ref{notation}, we collect the necessary notation. In Section~\ref{VinMethod}, we present the essential aspects of Vinogradov's method \cite{V} for estimating exponential sums. Our exposition is based primarily on \cite[Section~8]{IWKO}. In Section~\ref{sec:exp}, which constitutes the technical heart of the paper, we prove Proposition~\ref{expsumestfull}. In Section~\ref{CircleMethod}, we develop a variant of the one-frequency circle method and prove Proposition~\ref{expsumestfullforav}. Finally, in Section~\ref{sectionOrbit}, we establish Theorem~\ref{thm:mainq}.

\section*{Acknowledgments}
The first author would like to thank Borys Kuca for several useful discussions, as well as Nikos Frantzikinakis and B{\l}a{\.z}ej Wr{\'o}bel for their constant support and encouragement. Also, he is grateful to Trevor Wooley for their lectures during the Simons School on Discrete harmonic analysis and analytic number theory at the R\'enyi Institute in Budapest, Spring 2026, for clarifying certain aspects of Vinogradov's method of estimating exponential sums. These lectures offered great insights on how the method produces decay, which together with the method's exposition in Section~8.5 from \cite{IWKO}, proved to be invaluable. 

\section{Notation}\label{notation}
We now set up the notation that will be used throughout the article. 
\subsection{Basic notation}
We use the convention $\mathbb N \coloneqq \{0,1,2,\dots\}$, and the sets $\mathbb Z, \, \mathbb Q, \, \mathbb R$,  $\mathbb C$ and $\mathbb T \coloneqq \mathbb R / \mathbb Z$ have their standard meaning.
For every real number $x\in\mathbb R$, we use the floor, fractional part, and distance to the closest integer functions, defined by $\lfloor x \rfloor \coloneqq \max\{n\in\mathbb Z:n\le x\}$, $\{x\}\coloneqq x-\lfloor x\rfloor$ and $\|x\|\coloneq \min\{|x-n|:\,n\in\mathbb{Z}\}$, respectively. Throughout the paper the torus $\mathbb{T}$  is identified  with the unit interval $[-1,2/1,2) $, so that $\|\xi\|=|\xi|$ for every $\xi\in\mathbb{T}$ . 

For $N \in \mathbb R$ and $\mathbb{S}\subseteq \mathbb R$, we also define the sets
\begin{align*}
\mathbb{S}_{<N}&\coloneqq \mathbb{S}\cap(-\infty, N),  \quad \phantom{],} \mathbb{S}_{>N}\coloneqq \mathbb{S}\cap(N, \infty),\\
\mathbb{S}_{\le N}&\coloneqq \mathbb{S}\cap(-\infty, N],  \quad 
\phantom{),} \mathbb{S}_{\ge N}\coloneqq \mathbb{S}\cap[N, \infty),
\end{align*}
and for $N=0$ we let for convenience  $\mathbb{S}_+\coloneq \mathbb{S}_{>0}$. For $N \in \mathbb R_+$, we define $[N] \coloneqq \mathbb Z \cap (0,N]$. If $A$ is a finite set, its number of elements is denoted by $|A|$.

Finally, for any sequences $(a_m)_{m\in\mathbb Z}\subseteq\mathbb C$ and
$(b_m)_{m\in\mathbb Z}\subseteq\mathbb C$ and any integers $U,V\in\mathbb Z$ such
that $U< V$ we will use the following version of the summation by
parts formula
\begin{align}
\label{eq:31}
\sum_{m=U+1}^{V}a_mb_m=S_{V}b_{V}-\sum_{m=U+1}^{V-1}S_m(b_{m+1}-b_m), \quad \text{where} \quad S_m=\sum_{l=U+1}^ma_l.
\end{align}

\subsection{Asymptotic notation} Throughout the paper, $C\in\mathbb R_+$ denotes an absolute constant whose value may change from line to line. For two quantities $A,B\in\mathbb R_{\ge 0}$, we write $A \lesssim B$ or $B \gtrsim A$ if $A\le CB$ for some $C\in\mathbb R_+$ and $A\simeq B$ if $A \lesssim B\lesssim A$. We will use the symbols $\lesssim_{\delta}$ and $\simeq_\delta$ to emphasize that the implicit constant $C$ depends on a parameter $\delta$.
For two functions $f \colon X\to \mathbb C$ and $g \colon X\to \mathbb R_+$, we write $f = O(g)$ if there exists a constant $C>0$ such that $|f(x)| \le C g(x)$ for all $x\in X$.
\subsection{Euclidean spaces}
Let  $d\in\mathbb Z_+$,  the standard inner product and the corresponding Euclidean norm on $\mathbb R^d$ are denoted, respectively, by 
\begin{align*}
x\cdot\xi \coloneqq \sum_{i \in [d]} x_i \xi_i,
\quad \text{and} \qquad
|x|  \coloneq \sqrt{x\cdot x }
\end{align*}
for every $x=(x_1,\ldots, x_d)$ and
$\xi=(\xi_1, \ldots, \xi_d)\in\mathbb R^d$. 

For two vectors $x=(x_1,\ldots,x_d)$, $y=(y_1,\ldots,y_d)\in\mathbb R^d$, we define their Hadamard product by
\begin{align}
\label{eq:24}
x\odot y\coloneqq (x_1y_1,\ldots,x_dy_d)\in\mathbb R^d\text{,}
\end{align}
and note that
$(x\odot y)\cdot z=(x\odot z)\cdot y$ for every $x,y,z\in\mathbb R^d$.

\subsection{Fourier transform} Recall that $e (x)\coloneqq e^{2\pi i x}$  for every $x\in\mathbb{R}$.
The Fourier transform  of a function $f\in\ell^1(\mathbb Z)$ is defined by
\begin{align*}
\mathcal F_{\mathbb Z} [f](\xi) \coloneqq 
\sum_{n\in\mathbb Z} f(n) e (n \xi),\quad \xi\in\mathbb T\text{,}
\end{align*}
while the  inverse Fourier transform  of a function $g\in L^1(\mathbb T)$ is defined by
\begin{align*}
 \mathcal F_{\mathbb Z}^{-1} [g](n) \coloneqq \int_{\mathbb T} g(\xi) e (-x  \xi)d\xi,\quad n\in\mathbb Z.
\end{align*}
By Plancherel's theorem, the Fourier transform extends to a surjective isometry from $\ell^2(\mathbb Z)$ onto $L^2(\mathbb T)$, and $\|\mathcal F_{\mathbb Z}[f]\|_{L^2(\mathbb T)}=\|f\|_{\ell^2(\mathbb Z)}$ for every $f\in\ell^2(\mathbb Z)$.

\subsection{Variational seminorms}
For any $\mathbb I\subseteq \mathbb R$, any family
$(\mathfrak a_t)_{t\in\mathbb I}=(\mathfrak a_t: t\in\mathbb I)\subseteq \mathbb C$,
and any exponent $r \in [1, \infty)$, we define the $r$-variation seminorm of the family by
\begin{align}
\label{eq:18}
V^{r}( \mathfrak a_t: t\in\mathbb I) \coloneqq 
\sup_{J\in\mathbb Z_+} \sup_{\substack{t_{0}<\dotsb<t_{J}\\ t_{j}\in\mathbb I}}
\Big(\sum_{j \in [J]}  |\mathfrak a_{t_{j}}-\mathfrak a_{t_{j-1}}|^{r} \Big)^{1/r},
\end{align}
where the supremum is taken over all finite increasing sequences in
$\mathbb I$, and is set by convention to equal zero if
$|\mathbb I| \leq 1$. The ${V}^r$ norms are nonincreasing in $r$, and if 
$\mathbb I\subseteq \mathbb R$ is countable, we have
\begin{equation}
\label{eq:25}
{V}^{r}( \mathfrak a_t: t\in\mathbb I)\lesssim \Big(\sum_{t\in\mathbb I}|\mathfrak a_t|^r\Big)^{1/r}\text{.}
\end{equation}

\section{Vinogradov's method}\label{VinMethod}
In this section, we recall the key features of Vinogradov's method
\cite{V} and the Vinogradov mean value theorem, which will play an
essential role in our exponential sum estimates. We briefly introduce
the relevant notation and standard tools below. Our exposition is
based primarily on \cite[Section~8]{IWKO}.

\begin{proposition}\label{easyfactsJ}
Let $k,n\in\mathbb{Z}_+$ and $P\ge1$ and for every $y\in\mathbb R$ define the moment curve  
\begin{align*}
\mathfrak m_n(y)\coloneqq (y, y^2,\ldots, y^n)\in \mathbb R^n.
\end{align*}
For every $\boldsymbol{\lambda}=(\lambda_1,\dotsc,\lambda_n)\in\mathbb{Z}^n$, we define 
\begin{equation}\label{defJ}
J_{k,n}(P;\boldsymbol{\lambda})\coloneqq \int_{\mathbb{T}^n}\Big|\sum_{y\in[P]}e({\bm \alpha}\cdot \mathfrak m_n(y))\Big|^{2k}e\big(-{\bm \alpha}\cdot {\bm \lambda}\big)d\boldsymbol{\alpha}\text{,}
\end{equation}
and for convenience, we set $
J_{k,n}(P)\coloneqq J_{k,n}\big(P;(0,\dotsc,0)\big)$. Then the  following properties hold.
\smallskip
\begin{enumerate}[itemsep=5pt]
\item[\normalfont{(i)}] $
0\le J_{k,n}(P;\boldsymbol{\lambda})\le J_{k,n}(P)$.
\item[\normalfont{(ii)}] $
J_{k,n}(P;\boldsymbol{\lambda})=\big|\big\{(u_1,\dotsc, u_k,v_1,\dotsc, v_k)\in [P]^{2k}: \sum_{i\in[k]}(u_i^s-v_i^s)=\lambda_s\text{ for every } s\in[n]\big\}\big|$.
\item[\normalfont{(iii)}] $\sum_{\boldsymbol{\lambda}\in\mathbb{Z}^n}J_{k,n}(P;\boldsymbol{\lambda})=P^{2k}$.
\item[\normalfont{(iv)}] If $J_{k,n}(P; \boldsymbol{\lambda})\neq 0$, then for every $i\in[n]$ we have $
|\lambda_i|\le kP^i$, i.e.
\[
\supp(J_{k,n}(P;\cdot))\subseteq S_{k,n}(P)\coloneqq\big\{\boldsymbol{\lambda}\in\mathbb{Z}^n:\,|\lambda_i|\le kP^i\text{ for every }i\in[n]\big\}.
\]
\end{enumerate}
\end{proposition}
\begin{proof}
The proof is standard. Property \normalfont{(ii)} is obtained by expanding the power in \eqref{defJ} and exploiting orthogonality, and immediately yields every other property except for the second inequality in \normalfont{(i)} which is immediate from the definition.
\end{proof}
A crucial tool of Vinogradov's method is the  Vinogradov mean value theorem, which reads as follows.
\begin{theorem}\label{VMVT}
Let $k, m, n\in\mathbb Z_+$ be such that $k\ge n(n+m)$. Then for every $P\ge n^{n(1-1/n)^{-m}}$, we have
\begin{equation}\label{VMVTbounds}
J_{k,n}(P)\le 2^{4km}P^{2k-\frac{n(n+1)}{2}+\frac{n(n+1)}{2}(1-1/n)^m}\text{.}
\end{equation}
\end{theorem}
\begin{proof}
For the proof we refer to  \cite[Theorem~8.21]{IWKO}, see also \cite{V}.
\end{proof}

As discussed in the introduction, for our application of the method, the explicit dependence of the bounds in \eqref{VMVTbounds} on the parameters $k,n$ is indispensable.

We briefly now describe the general main set up of Vinogradov's method for estimating exponential sums, inherently relying on estimates for the number of solutions for the Vinogradov system of Diophantine equations, namely on bounds for $J_{k,n}(P)$. 

\subsection{Step 1: Doubling the variables by a small shift}
For any $N\in\mathbb Z_+$, $P\in [1, N^{1/2}]$ and any function $\mathfrak h:\mathbb R_+\to \mathbb R$, we have 
\begin{equation}\label{shiftq}
\Big|\sum_{a=N+1}^{2N}e(\mathfrak h(a)\xi)\Big|\le \frac{1}{P^2}\sum_{a=N+1}^{2N}|U_P(a)|+2P^2,
\end{equation}
where
\[
U_P(a)\coloneqq\sum_{y\in[P]}\sum_{z\in[P]}e(\mathfrak h(a+yz)\xi).
\]
To see this, we write
\begin{equation}\label{shift1q}
\begin{split}
\frac{1}{P^2}\sum_{a=N+1}^{2N}U_P(a)
&=\sum_{a=N+1}^{2N}e(\mathfrak h(a)\xi)+\frac{1}{P^2}\sum_{y\in[P]}\sum_{z\in[P]}\bigg(\sum_{a=N+1+yz}^{2N+yz}e(\mathfrak h(a)\xi)-\sum_{a=N+1}^{2N}e(\mathfrak h(a)\xi)\bigg)
\\
&=\sum_{a=N+1}^{2N}e(\mathfrak h(a)\xi)+\frac{1}{P^2}\sum_{y\in[P]}\sum_{z\in[P]}\bigg(\sum_{a=2N+1}^{2N+yz}e(\mathfrak h(a)\xi)-\sum_{a=N+1}^{N+yz}e(\mathfrak h(a)\xi)\bigg).
\end{split}
\end{equation}
By the triangle inequality we estimate the last term as follows 
\begin{equation}\label{shift2q}
\begin{split}
&\bigg|\frac{1}{P^2}\sum_{y\in[P]}\sum_{z\in[P]}\bigg(\sum_{a=2N+1}^{2N+yz}e(\mathfrak h(a)\xi)-\sum_{a=N+1}^{N+yz}e(\mathfrak h(a)\xi)\bigg)\bigg|\le\frac{1}{P^2}\sum_{y\in[P]}\sum_{z\in[P]} 2yz\le 2P^2.
\end{split}
\end{equation}
By \eqref{shift1q} and \eqref{shift2q} we obtain \eqref{shiftq}, as desired.
\subsection{Step 2: Taylor approximation} We additionally assume that  $\mathfrak h:\mathbb R_+\to \mathbb R$ is smooth, and we further seek to estimate $|U_P(a)|$ for $a\in (N,2N]$.  By Taylor
expanding up to order $n\in\mathbb Z_+$, we find an intermediate point  $\theta_{a, n,yz}\in(N+1,3N)$ such that $\mathfrak h(a+yz)=\mathfrak h_n(a, yz)+\mathfrak r_n(a, yz)$, where
\[
\mathfrak h_n(a, yz)\coloneqq\sum_{j=0}^n\frac{\mathfrak h^{(j)}(a)}{j!}(yz)^{j}
\qquad \text{ and } \qquad \mathfrak r_n(a, yz) \coloneqq \frac{\mathfrak h^{(n+1)}(\theta_{a,n, yz})}{(n+1)!}(yz)^{n+1}.
\]
If we further let
\[
V_P(a)\coloneqq\sum_{y\in[P]}\sum_{z\in[P]}e\big(\mathfrak h_n(a, yz)\xi\big),
\]
and use the trivial estimate $|e(x)-1|\le 2\pi |x|$ together with the Taylor expansion, we obtain
\begin{equation}\label{TaylorRemainderq}
\begin{split}
|U_P(a)|\le&|V_P(a)|+\Big|\sum_{y\in[P]}\sum_{z\in[P]}e\big(\mathfrak h(a+yz)\xi\big)-e\big(\mathfrak h_n(a, yz)\xi\big)\Big|
\\
\le&|V_P(a)|+2\pi \sum_{y\in[P]}\sum_{z\in[P]}\bigg|\frac{\xi \mathfrak h^{(n+1)}(\theta_{a, n, yz})}{(n+1)!}(yz)^{n+1}\bigg|\\
\le&|V_P(a)|+2\pi |\xi| P^{2(n+2)}\sup_{\theta\in(N+1, 3N)}\bigg|\frac{\mathfrak h^{(n+1)}(\theta)}{(n+1)!}\bigg|.
\end{split}
\end{equation}
Combining \eqref{shiftq} with \eqref{TaylorRemainderq}, we conclude that
\begin{align}
\label{eq:1}
\Big|\sum_{a=N+1}^{2N}e(\mathfrak h(a)\xi)\Big|\le \frac{1}{P^2}\sum_{a=N+1}^{2N}|V_P(a)|+2\pi |\xi| P^{2(n+1)}N\sup_{\theta\in(N+1, 3N)}\bigg|\frac{\mathfrak h^{(n+1)}(\theta)}{(n+1)!}\bigg|+2P^2.
\end{align}

\subsection{Step 3: The moment curve and Vinogradov's counting function revelation} For every $j\in[n]$ we let 
$B_j\coloneqq B_j(a)\coloneqq\frac{\mathfrak h^{(j)}(a)}{j!}\xi $, and we form a  vector $B\coloneqq(B_1,\ldots, B_n)\in \mathbb R^n$. By Holder's inequality, for every $k\in\mathbb Z_+$, we reveal the moment curve and obtain
\begin{equation}
\label{eq:2}
|V_P(a)|\le P^{\frac{2k-1}{2k}}\Big(\sum_{y\in[P]}\Big|\sum_{z\in[P]}e\big(B\cdot \mathfrak m_n(yz)\big)\Big|^{2k}\Big)^{\frac{1}{2k}},
\end{equation}
since $\mathfrak h_n(a, yz)\xi=\mathfrak h(a)\xi+ B\cdot \mathfrak m_n(yz)$.
Having identified the moment curve, we now see that the Vinogradov counting function $J_{k, n}(P;{\bm \lambda})$ arises naturally. Namely, by squaring and \eqref{eq:24} we have the identity
\[
\Big|\sum_{z\in[P]}e\big(B\cdot \mathfrak m_n(yz)\big)\Big|^2=
\sum_{u, v\in[P]}e\Big(\big(B\odot \mathfrak m_n(y)\big) \cdot\big(\mathfrak m_n(u)-\mathfrak m_n(v)\big)\Big),
\]
since $B\cdot \mathfrak m_n(yz)=\big(B\odot \mathfrak m_n(y)\big)\cdot \mathfrak m_n(z)$. Combining this identity 
with \eqref{eq:2} yields the following bound
\begin{align*}
|V_P(a)|^{2k}\le
&P^{2k-1}\sum_{y\in[P]}\bigg(\sum_{u, v\in[P]}e\Big(\big(B\odot \mathfrak m_n(y)\big) \cdot\big(\mathfrak m_n(u)-\mathfrak m_n(v)\big)\Big)\bigg)
^{k}
\\
=&P^{2k-1}\sum_{y\in[P]}\sum_{\substack{u_l,v_l\in[P]\\ l\in[k]}}e\Big(\big(B\odot \mathfrak m_n(y)\big) \cdot\big(\sum_{l\in[k]}\mathfrak m_n(u_l)-\mathfrak m_n(v_l)\big)\Big)
\\
=&P^{2k-1}\sum_{y\in[P]}\sum_{\boldsymbol{\lambda}\in\mathbb{Z}^n}J_{k,n}(P;{\bm \lambda})e\Big(\big(B\odot \mathfrak m_n(y)\big) \cdot{\bm \lambda}\Big).
\end{align*}
Noting that $\big(B\odot \mathfrak m_n(y)\big) \cdot{\bm \lambda}=\big(B\odot {\bm \lambda}\big) \cdot \mathfrak m_n(y)$ and using H\"older's inequality and all four assertions of Proposition~$\ref{easyfactsJ}$, we may further write
\begin{equation}\label{W1longq}
\begin{split}
|V_P(a)|^{2k}\le&P^{2k-1}\sum_{\boldsymbol{\lambda}\in\mathbb{Z}^n}J_{k,n}(P;\boldsymbol{\lambda})\Big|\sum_{y\in[P]}e\big(\big(B\odot {\bm \lambda}\big) \cdot \mathfrak m_n(y)\big)\Big|
\\
\le&P^{2k-1}\Big(\sum_{\boldsymbol{\lambda}\in\mathbb{Z}^n}J_{k,n}(P;\boldsymbol{\lambda})\Big)^{\frac{2k-1}{2k}}
\cdot \Big(\sum_{\boldsymbol{\lambda}\in\mathbb{Z}^n}J_{k,n}(P;\boldsymbol{\lambda})\Big|\sum_{y\in[P]}e\big(\big(B\odot {\bm \lambda}\big) \cdot \mathfrak m_n(y)\big)\Big|^{2k}\Big)^{\frac{1}{2k}}
\\
\le& P^{4k-2}
J_{k,n}(P)^{\frac{1}{2k}}\cdot \Big(\sum_{\boldsymbol{\lambda}\in S_{k,n}(P)}\Big|\sum_{y\in[P]}e\big((B\odot {\bm \lambda}) \cdot \mathfrak m_n(y)\big)\Big|^{2k}\Big)^{\frac{1}{2k}}.
\end{split}
\end{equation}
Proceeding in a similar manner we may bound the last factor above by noting that
\begin{equation}\label{finalW1q}
\begin{split}
\sum_{\boldsymbol{\lambda}\in S_{k,n}(P)}\Big|\sum_{y\in[P]}e\big((B\odot {\bm \lambda}) \cdot \mathfrak m_n(y)\big)\Big|^{2k}
=&\sum_{\boldsymbol{\lambda}\in S_{n,k}(P)}\sum_{\substack{u_l,v_l\in[P]\\ l\in[k]}}e\Big((B\odot {\bm\lambda}) \cdot\big(\sum_{l\in[k]}\mathfrak m_n(u_l)-\mathfrak m_n(v_l)\big)\Big)
\\
=&\sum_{\boldsymbol{\lambda}\in S_{n,k}(P)}\sum_{\boldsymbol{\mu}\in\mathbb{Z}^n}J_{k, n}(P; \boldsymbol{\mu})e\big((B\odot {\bm\lambda}) \cdot{\bm\mu}\big)
\\
\le&\sum_{\boldsymbol{\mu}\in\mathbb{Z}^n}J_{k, n}(P; \boldsymbol{\mu})\Big|\sum_{\boldsymbol{\lambda}\in S_{n,k}(P)}e\big((B\odot {\bm\lambda}) \cdot{\bm\mu}\big)\Big|
\\
\le& J_{k,n}(P)\prod_{j\in[n]}\sum_{|\mu_j|\le kP^j}\Big|\sum_{|\lambda_j|\le kP^j}e\big(B_j\lambda_j\mu_j\big)\Big|.
\end{split}
\end{equation}
Combining \eqref{W1longq} with \eqref{finalW1q} and raising once again to the $2k$-th power yields
\begin{equation}\label{beforenotationq}
|V_P(a)|^{4k^2}\le P^{8k^2-4k}J_{k,n}(P)^2\prod_{j\in[n]}\sum_{|\mu_j|\le kP^j}\Big|\sum_{|\lambda_j|\le kP^j}e\big(B_j(a)\lambda_j\mu_j\big)\Big|.
\end{equation}

\subsection{Step 4: Double exponential sum estimates}
For every $\alpha\in\mathbb{R}$ and $X\in\mathbb{Z}_+$ we define
\[
D(\alpha,X)\coloneqq X^{-2}\sum_{|m|\le X}\Big|\sum_{|n|\le X}e(\alpha m n)\Big|.
\]
and letting $\Delta(a)\coloneqq\prod_{j\in[n]}D(B_j(a),kP^j)$, we see that
\begin{equation}\label{deltacalculationq}
\begin{split}
\prod_{j\in[n]}\sum_{|\mu_j|\le kP^j}\Big|\sum_{|\lambda_j|\le kP^j}e\big(B_j(a)\lambda_j\mu_j\big)\Big|&=\prod_{j\in[n]}k^2P^{2j}D(B_j(a),kP^j)\\
&=\Delta(a)\prod_{j\in[n]}k^2P^{2j}
=\Delta(a) k^{2n}P^{n(n+1)}.
\end{split}
\end{equation}
Taking into account \eqref{deltacalculationq}, the estimate \eqref{beforenotationq} can be rewritten as
\begin{equation}\label{W1closeq}
|V_P(a)|\le P^{2-\frac{1}{k}}J_{k,n}(P)^{\frac{1}{2k^2}}\Delta(a)^{\frac{1}{4k^2}} k^{\frac{n}{2k^2}}P^{\frac{n(n+1)}{4k^2}}.
\end{equation}
Proceeding as in \cite[pp. 224-225]{IWKO}, for every $X\ge 3$ and $\alpha\in\mathbb{R}$,   we have that
\[
D(\alpha,X)\le 10\min\big\{1,\big(|\alpha|+|\alpha|^{-1}X^{-2}\big)\log(3X)\big\}.
\]
Taking $P\ge 3$ (which we will always be able to arrange), we may apply this estimate to obtain
\begin{equation}\label{deltafirstestnewq}
\Delta(a)\le\Big(\prod_{j\in I}10\big(|B_j(a)|+|B_j(a)|^{-1}(kP^j)^{-2}\big)\log(3kP^j)\Big)
\cdot\Big(\prod_{j\in[n]\setminus I}10\Big)\le \Delta_I(a),
\end{equation}
for any subset $I\subseteq [n]$, where
\begin{align}
\label{eq:3}
\Delta_I(a)\coloneqq\big(10n\log(3kP)\big)^n\prod_{j\in I}\big(|B_j(a)|+|B_j(a)|^{-1}P^{-2j}\big).
\end{align}
Now gathering \eqref{eq:1}, \eqref{W1closeq}, \eqref{deltafirstestnewq} and \eqref{eq:3}, we conclude that for $P\in[3, N^{1/2}]$ and $k, n\in\mathbb Z_+$, we have
\begin{align}
\label{eq:4}
\begin{split}
\Big|\frac{1}{N}\sum_{a=N+1}^{2N}e(\mathfrak h(a)\xi)\Big|\le& \frac{1}{N}\sum_{a=N+1}^{2N}J_{k,n}(P)^{\frac{1}{2k^2}}\Delta_I(a)^{\frac{1}{4k^2}} k^{\frac{n}{2k^2}}P^{\frac{n(n+1)}{4k^2}-\frac{1}{k}}\\
&+2\pi |\xi| P^{2(n+1)}\sup_{\theta\in(N+1, 3N)}\bigg|\frac{\mathfrak h^{(n+1)}(\theta)}{(n+1)!}\bigg|+\frac{2P^2}{N}.
\end{split}
\end{align}
We will apply \eqref{eq:4} with $\mathfrak h(x)=\exp\big((\log x)^c\big)$ whenever $c\in(1, 8/7)$ and a suitable $I\subseteq [n]$.

We conclude with a high-level overview of the proof strategy for Proposition~\ref{expsumestfull}. A certain degree of imprecision is unavoidable here, since our aim is to provide a ``global'' technical intuition for the competing effects of the various parameters. Because these parameters are numerous, and because the constants in our estimates must remain independent of them, we opted for a presentation in which the relevant choices are made explicitly from the outset. We aim to convince the reader that these choices are natural and encourage them to refer back to this discussion while following the proof.

With the right-hand side of \eqref{eq:4} as our starting point, it remains to choose $(P,n,m,k)$ and a set $I\subseteq[n]$ so that all three summands exhibit some subpolynomial decay. The third parameter does not appear explicitly in \eqref{eq:4}, and we simply mean here that one must also make an admissible choice of $m$ when applying Theorem~\ref{VMVT}.

The only factor capable of contributing decay in the first summand is $\Delta_I(a)^{\frac{1}{4k^2}}$; see, for example, the calculation in \eqref{input1qq}. This is not a feature unique to the particular form of the Vinogradov mean value theorem that we use. For exceptionally low frequencies $|\xi|=h(N)^{-1}N^{\rho}$, where $\rho>0$, the second factor in the product defining $\Delta_I(a)$ becomes rather problematic, since
\begin{equation}\label{heuristics}
|B_j(a)|^{-1}P^{-2j}=|\xi|^{-1}\bigg|\frac{h^{(j)}(a)}{j!}\bigg|^{-1}P^{-2j}\simeq N^{-\rho}(N/P^2)^j\text{,}
\end{equation}
where we used the heuristic $|h^{(j)}(a)|/j!\simeq N^{-j}h(N)$ for $a\simeq N$ and $j\in[c(\log N)^{c-1}]$; see Lemma~\ref{basicprop}. By \eqref{heuristics}, in order to obtain decay, one cannot allow $N/P^2$ to grow polynomially in $N$. This already forces us to depart completely from Karatsuba's strategy; see Theorem~1 in \cite{Kar}, where the shift parameter is chosen to be polynomially smaller than $\sqrt{N}$.

We are therefore led to consider choices of the form $P\simeq N^{1/2}e^{-(\log N)^{\tau}}$, $\tau\in(0,1)$, since this is the absolute upper limit on the growth of $P$ imposed by the third summand in \eqref{eq:4}. This, in turn, makes the second summand problematic, since $P$ is raised to a large power. In contrast to the third summand, however, the second summand improves as the order of the Taylor approximation increases, that is, as $n$ becomes larger.

The natural choice for the order of magnitude of $n$ is $(\log N)^{c-1}$. Loosely speaking, this matches the ``local polynomial growth'' of $h(x)=x^{(\log x)^{c-1}}$ at scale $N$, as well as the amount of lower-bound information on the derivatives available to us; see Lemma~\ref{basicprop}. Unfortunately, the large size of the shift parameter $P$ forces us to choose a Taylor approximation of substantially higher order in order to control the second term in \eqref{eq:4}. As a result, the choice of $k$ must also be large; see the conditions of Theorem~\ref{VMVT}. This, in turn, diminishes any saving one hopes to obtain from $\Delta_I(a)^{\frac{1}{4k^2}}$. Moreover, the amplification of the decay coming from the fact that $\Delta_I(a)$ involves a product of $|I|$ small factors does not improve once $n\gtrsim (\log N)^{c-1}$ since one cannot use lower-bound control on derivatives of order substantially larger than the ``local polynomial growth'' of $h$: we must choose $I\subseteq [c(\log N)^{c-1}]$.

Nevertheless, after departing from the natural choice of $n$ and fixing the forms of $P$ and $n$ as
$$
P=\lfloor N^{1/2}e^{-(\log N)^{\tau}}\rfloor\text{,}\quad n=(\log N)^{\sigma}\quad\text{with $\tau\in(0,1)$ and $\sigma>0$,}
$$
a careful analysis of the second term in \eqref{eq:4} leads to the restriction $\tau>c-\sigma$ in order to obtain subpolynomial decay. For $\Delta_I(a)$, choosing the largest possible $I\subseteq [c(\log N)^{c-1}]$ for which the second term in the product defining $\Delta_I$ is dominant (which, unsurprisingly, can be arranged for a positive proportion of $[c(\log N)^{c-1}]$ in the exceptionally small minor arc frequency regime), one obtains the second restriction $\tau<2-c$ in order to obtain decay for $\Delta_I(a)$, which is of the form $e^{-\rho' (\log N)^c}$. These two restrictions imply that $\sigma>2c-2$, and the larger $n$ is, the more severe the deterioration of the saving for $\Delta_I(a)^{\frac{1}{4k^2}}$. We therefore choose
$$
n\simeq(\log N)^{2c-2+\delta}\text{,}\quad m\simeq n(\log\log N)\text{,}\quad k=2mn\text{,}
$$
and obtain
$$
\Delta_I(a)^{\frac{1}{4k^2}}\lesssim e^{-\frac{\rho''(\log N)^c}{(\log\log N)^2(\log N)^{8c-8+4\delta}}}\lesssim e^{-\rho''(\log N)^{8-7c-8\delta}}\text{,}
$$
which gives the required saving, provided that $\delta$ is chosen appropriately and $c<8/7$.

\section{Exponential sum estimates: Proof of Proposition~\ref{expsumestfull}}
\label{sec:exp}
We begin by stating the following exponential sum estimate, essentially due to Karatsuba~\cite{Kar}, see also Br{\"u}dern and Perelli~\cite[Theorem 2]{ApKar}. Here and throughout the work, we write $h(x)\coloneqq\exp\big((\log x)^c\big)$.

\begin{proposition}\label{prop:BP}
Let $c\in(1,3/2)$, $\alpha\in(0,1)$ and $\beta\in(0,c-1)$. Then there exist $\kappa=\kappa(c,\alpha,\beta)\in\mathbb R_+$ and $C=C(c,\alpha,\beta)\in\mathbb R_+$ such that for every $N\in\mathbb{Z}_+$ and $\xi\in\mathbb{R}$ with $h(2N)^{-\alpha}\le|\xi|\le h(N)^{\beta}$, we have
\begin{equation}
\label{eq:5}
\Big|\sum_{N<n\le 2N}e\big(h(n)\xi \big)\Big|\le CNe^{-\kappa(\log N)^{3-2c}}.
\end{equation}
\end{proposition}
For our purposes, the lower range of~$|\xi|$ in Proposition~\ref{prop:BP} is not sufficient. The proof of Proposition~\ref{expsumestfull} reduces to treating exceptionally low frequencies, while maintaining bounds of the same quantitative nature as in~\eqref{eq:5}.

Before proceeding with the proof of Proposition~$\ref{expsumestfull}$, let us collect the following two technical lemmas.

\begin{lemma}[Derivative estimates]\label{basicprop}
For every $c\in(1,3/2)$ there exists $x_0=x_0(c)\in\mathbb R_+$ such that for every real number $x\ge x_0$ and every natural number $s\le c(\log x)^{c-1}$ we have
\begin{equation}\label{lowerboundder}
\frac{h^{(s)}(x)}{s!}\ge \frac{h(x)}{2x^s}\text{.}
\end{equation}
Also, for every $c\in(1,\infty)$, $N\in\mathbb{Z}_{\ge 10}$, and $s\in\mathbb{Z}_+$, we have that for every $x\in(N,2N]$ the following estimate holds
\begin{equation}\label{upperboundder}
\bigg|\frac{h^{(s)}(x)}{s!}\bigg|\le 2^{s+1}N^{-s}h(3N)\text{.}
\end{equation}
\end{lemma}
\begin{proof}
For a detailed proof, we refer the reader to \cite[Lemmas~14 and~15, Section~8]{ApKar}.
\end{proof}
\begin{lemma}[Ratio estimate]
\label{help}
For real numbers $c\in(1,\infty)$ and $x,y\in[1,\infty)$ with $x\ge y$, we have
\[
\frac{h(yx)}{h(x)}\le \exp\big(c2^{c-1}(\log y)(\log x)^{c-1}\big)\text{.}
\]
\end{lemma}
\begin{proof}
The proof is a straightforward application of the mean value theorem. For $y=1$ the estimate clearly holds, and for $y>1$, by the mean value theorem there exists $\xi\in(\log x ,\log x +\log y )\subseteq(\log x ,2\log x )$ such that $(\log x +\log y )^c-(\log x )^c=(\log y) c\xi^{c-1}$. The bound is now immediate, since
\[
\frac{h(yx)}{h(x)}=\exp\big((\log x +\log y )^c-(\log x )^c\big)=\exp\big((\log y) c\xi^{c-1}\big)\le \exp\big(c2^{c-1}(\log y) (\log x )^{c-1}\big).
\]
This completes the proof of the lemma.
\end{proof}

We are now ready to give a proof of Proposition~$\ref{expsumestfull}$.
\begin{proof}[Proof of Proposition~$\ref{expsumestfull}$]
We fix $c\in(1,8/7)$, $\rho\in(0,\infty)$, and $\theta\in(0,c-1)$, and
allow all implicit constants to depend on these parameters. By the
trivial estimate, it clearly suffices to establish the result for
sufficiently large $N\in\mathbb Z_+$. Moreover, without loss
of generality, we may assume that $\xi>0$, since the case $\xi<0$
follows by complex conjugation. We will proceed in a few steps.
\smallskip
\paragraph{\textbf{Step 1: Applying Karatsuba's estimate in \eqref{eq:5}}}
From now on, we may assume that
\begin{align}
\label{eq:6}
h(N)^{-1}N^{\rho}\le\xi\le h(N)^{-1+\frac{c-1}{4}}.
\end{align}
Otherwise, if $\xi$ satisfies $h(N)^{-1+\frac{c-1}{4}}\le \xi\le h(N)^{\theta}$, then Proposition~\ref{prop:BP} applies directly, yielding the existence of constants $C=C(c,\theta)\in\mathbb R_+$ and $\kappa=\kappa(c,\theta)\in\mathbb R_+$ such that the following bound
\begin{equation}\label{firstequationq}
\Big|\frac{1}{N}\sum_{N< n\le 2N}e( h(n)\xi)\Big|\le C e^{-\kappa(\log N) ^{3-2c}}\le C e^{-(\log N) ^{(3-2c)/2}}
\end{equation}
holds for sufficiently large integers $N\in\mathbb Z_+$.
\smallskip
\paragraph{\textbf{Step 2: Applying Vinogradov's method}}
Fix
$P\coloneqq\lfloor N^{1/2}e^{-(\log N)^{\tau}}\rfloor$ with
\[
\tau\coloneqq 2-c-\frac{\delta}{2}\text{,}\quad\text{where}\quad \delta\coloneqq\frac{8-7c}{16}.
\]
Note that $\delta\in(0,1/16)$, since $c\in(1,8/7)$, and consequently  $\tau\in(2-c-1/32,2-c)$. Also, we have that 
\[
\tau=2-c-\frac{\delta}{2}>2-\frac{8}{7}-\frac{1}{32}>\frac{6}{7}-\frac{1}{32}>0.
\]
Hence, $P\le N^{1/2}$ and, for sufficiently large integers $N\in\mathbb Z_+$, we also have $P\ge N^{1/3}$. We now define $\varepsilon\coloneqq \frac{\delta}{4}\in(0, 1/64)$, and the three additional integers by setting
\begin{equation}\label{choiceparametersq}
n\coloneqq\lfloor 10(\log N)^{2(c-1)+\delta}\rfloor, \quad\text{and}\quad
m\coloneqq n\big\lfloor (1-\varepsilon)(3-2c-\delta-\varepsilon)(\log\log N)\big\rfloor,\quad\text{and}\quad
k\coloneqq 2mn.
\end{equation}
We immediately note that $3-2c-\delta-\varepsilon=3-2c-\frac{5\delta}{4}>3-\frac{16}{7}-\frac{5}{64}=\frac{5}{7}-\frac{5}{64}>0$ since $c\in(1,8/7)$. Hence, for sufficiently large $N\in\mathbb Z_+$, the quantities $k$, $n$, and $m$ are positive integers.

 Consequently, for sufficiently large integers $N\in\mathbb Z_+$, we may apply the bound in~\eqref{eq:4}, with $\mathfrak h(x)=h(x)\coloneqq\exp\big((\log x)^c\big)$, the integers defined in~\eqref{choiceparametersq}, and the frequency $\xi$ satisfying~\eqref{eq:6}, to estimate the exponential sum appearing in Proposition~\ref{expsumestfull}.

\smallskip
\paragraph{\textbf{Step 3: Handling the error terms in \eqref{eq:4}}}
Since $\tau>0$, we note that 
\begin{equation}\label{Boundapproxq}
P^2\le Ne^{-2(\log N)^{\tau}}\le Ne^{-(\log N)^{\tau}}.
\end{equation} 
Thus we may focus on estimating the term involving the derivative  in \eqref{eq:4}. Indeed, for sufficiently large integers $N\in\mathbb Z_+$, it follows from~\eqref{eq:6} that $|\xi|\le h(N)^{-1+\frac{c-1}{4}}\le 1$. Moreover, by Lemma~\ref{basicprop}, we obtain
\begin{align}
\label{eq:7}
\begin{split}
|\xi|P^{2(n+1)}
\sup_{\theta\in(N+1,3N)}
\bigg|\frac{h^{(n+1)}(\theta)}{(n+1)!}\bigg|
&\le
2^{n+2}h(6N)N^{-(n+1)}P^{2(n+1)}\\
&\le e^{12(\log N)^{2c-2+\delta}}e^{(\log N)^c}e^{16(\log N)^{c-1}}e^{-2(\log N)^{\tau}(n+1)}\\
&\le e^{100(\log N)^{2c-2+\delta}}e^{(\log N)^c}e^{-2(\log N)^{\tau+2c-2+\delta}}.
\end{split}
\end{align}
In passing to the second line we used the bound
$h(6N)\le h(N)e^{16(\log N)^{c-1}}$ furnished by Lemma~\ref{help}, \eqref{Boundapproxq}, and $2^{n+2}\le e^{12(\log N)^{2c-2+\delta}}$.

Furthermore, for sufficiently large integers $N\in\mathbb Z_+$, we note that
\begin{equation}
\label{eq:8}
\begin{split}
&e^{100(\log N)^{2c-2+\delta}+(\log N)^c-2(\log N)^{\tau+2c-2+\delta}}\le e^{-(\log N)^{\tau+2c-2+\delta}}
\\
\iff&100(\log N)^{2c-2+\delta}+(\log N)^c\le(\log N)^{\tau+2c-2+\delta}
\\
\iff&100(\log N)^{-\tau}+(\log N)^{2-c-\tau-\delta}\le 1\text{,}
\end{split}
\end{equation}
since $2-c-\tau-\delta=2-c-\big(2-c-\frac{\delta}{2}\big)-\delta=-\frac{\delta}{2}<0$. Combining \eqref{Boundapproxq} with \eqref{eq:7} and \eqref{eq:8}, we conclude that the error terms in~\eqref{eq:4}, namely, the last two summands in its right side, admit the following bound
\begin{align}
\label{eq:9}
2\pi |\xi|P^{2(n+1)}
\sup_{\theta\in(N+1,3N)}
\bigg|\frac{h^{(n+1)}(\theta)}{(n+1)!}\bigg|
+\frac{2P^2}{N}\le 4\pi e^{-(\log N)^{\tau}}\text{.}
\end{align}

\paragraph{\textbf{Step 4: Applying Vinogradov's mean value theorem}}
Recalling the definitions of the integers $k$, $m$, and $n$ from~\eqref{choiceparametersq}, we note that, for sufficiently large integers $N\in\mathbb Z_+$, we have $m\ge n$ and $n(n+m)\le 2mn=k$. Hence, the first hypothesis of Theorem~\ref{VMVT} is satisfied. To verify the second, it suffices to show, for sufficiently large integers $N\in\mathbb Z_+$, that
\begin{align}
\label{eq:10}
P\ge n^{n(1-1/n)^{-m}}.
\end{align}
Taking $N\in\mathbb Z_+$ sufficiently large, we may assume that
$P\ge N^{1/3}$ as well as that
$n=\lfloor 10(\log N)^{2c-2+\delta}\rfloor\ge \frac{64}{8-7c}=\varepsilon^{-1}$, 
$ \log N\ge 10$, and $ (\log N)^{\varepsilon}\ge 60(\log\log N)$.  It
suffices to check that $N^{1/3}\ge n^{n(1-1/n)^{-m}}$. Indeed, using $\frac{n-1}{n}>1-\varepsilon$, note that
\[
(1-1/n)^{-m}=\Big(1+\frac{1}{n-1}\Big)^m\le e^{\frac{m}{n-1}}\le e^{\frac{1}{1-\varepsilon}\cdot \frac{m}{n}}.
\]
Since $2c-2+\delta=2(c-1)+\delta\le \frac{2}{7}+\frac{1}{16}\le 1$ and $\log 10 \le \log\log N$, we obtain
\begin{equation*}
\log(n)n\le
10(\log N)^{2c-2+\delta}\big(\log(10)+(2c-2+\delta)\log\log N\big)\le20(\log N)^{2c-2+\delta}\log\log N.
\end{equation*}
Combining these two bounds and taking into account \eqref{choiceparametersq}, we may write that
\begin{align*}
\log(n)n(1-1/n)^{-m}\le& 20(\log N)^{2c-2+\delta}(\log\log N) e^{\frac{1}{1-\varepsilon}\cdot\frac{m}{n}}
\\
\le&
20(\log N)^{2c-2+\delta}(\log\log N) e^{(3-2c-\delta-\varepsilon)(\log\log N)}
\\
=&20(\log N)^{2c-2+\delta}(\log\log N) (\log N)^{3-2c-\delta-\varepsilon}
\\
=&20(\log\log N)(\log N)^{1-\varepsilon} \le \frac{1}{3}\log N\le \log P,
\end{align*}
since $20(\log\log N)(\log N)^{1-\varepsilon} \le \frac{1}{3}\log N\iff
60(\log\log N)\le (\log N)^{\varepsilon}$. This verifies \eqref{eq:10} as desired.

Thus, for sufficiently large integers $N\in\mathbb Z_+$, Theorem~\ref{VMVT} is applicable and yields
\[
J_{k,n}(P)\le 2^{4km}P^{2k-\frac{n(n+1)}{2}+\frac{n(n+1)}{2}(1-1/n)^m}\text{.}
\]
Using this bound, we estimate each summand in~\eqref{eq:4}. For every integer $a\in(N,2N]$ we have
\begin{equation}\label{input1qq}
\begin{split}
J_{k,n}(P)^{\frac{1}{2k^2}}\Delta_I(a)^{\frac{1}{4k^2}} k^{\frac{n}{2k^2}}P^{\frac{n(n+1)}{4k^2}-\frac{1}{k}}
\le& 2^{\frac{2m}{k}}\Delta_I(a)^{\frac{1}{4k^2}} k^{\frac{n}{2k^2}} P^{\frac{1}{4k^2}n(n+1)(1-1/n)^m} 
\\
\le&2e\Delta_I(a)^{\frac{1}{4k^2}} P^{\frac{1}{4k^2}n(n+1)(1-1/n)^m},
\end{split}
\end{equation}
since $k=2mn$ and consequently  $k^{\frac{n}{2k^2}}\le k^{\frac{1}{k}}\le e$ and $2^{\frac{2m}{k}}\le 2$.

Taking $N\in\mathbb Z_+$ sufficiently large, we have that $m\ge n \ge 5(\log N)^{2c-2+\delta}$, and using 
$(1-1/n)^{m}\le e^{-\frac{m}{n}}$, we may write
\begin{align*}
\frac{1}{4k^2}n(n+1)(1-1/n)^m\le& \frac{2n^2}{16m^2n^2}e^{-\frac{m}{n}}\\
\le&\frac{e}{8m^2}e^{-(1-\varepsilon)(3-2c-\delta-\varepsilon)\log\log N}
\\
\le& n^{-2}(\log N)^{-(1-\varepsilon)(3-2c-\delta-\varepsilon)}
\\
\le& (\log N)^{4-4c-2\delta-(1-\varepsilon)(3-2c-\delta-\varepsilon)}\text{,}
\end{align*}
which implies that
\begin{equation}\label{input2qq}
P^{\frac{1}{4k^2}n(n+1)(1-1/n)^m}\le N^{(\log N)^{4-4c-2\delta-(1-\varepsilon)(3-2c-\delta-\varepsilon)}}=e^{(\log N)^{5-4c-2\delta-(1-\varepsilon)(3-2c-\delta-\varepsilon)}}\le e,
\end{equation}
since $5-4c-2\delta-(1-\varepsilon)(3-2c-\delta-\varepsilon)<(2-2c)-\delta+2\varepsilon=(2-2c)-\delta+\frac{\delta}{2}<0$.
Thus, by~\eqref{input1qq} and~\eqref{input2qq}, for sufficiently large integers $N\in\mathbb Z_+$ and every integer $a\in(N,2N]$, each summand in~\eqref{eq:4} satisfies the following estimate
\begin{equation}
\label{eq:11}
J_{k,n}(P)^{\frac{1}{2k^2}}\Delta_I(a)^{\frac{1}{4k^2}} k^{\frac{n}{2k^2}}P^{\frac{n(n+1)}{4k^2}-\frac{1}{k}}
\le e^3\Delta_I(a)^{\frac{1}{4k^2}}.
\end{equation}
\paragraph{\textbf{Step 5: Derivative estimates}}
Let $F= F(N,\xi)\coloneqq h(N)\xi$.  By Lemma~$\ref{basicprop}$ there exists $x_0(c)\in\mathbb R_+$ such that for every $N\ge\max\{ x_0(c),10\}$, every $a\in(N,2N]$ and every integer
$1\le j\le c(\log N)^{c-1}\le n$, we have
\begin{equation}\label{boundcoefq}
2^{-j-1}F\cdot  N^{-j}\le |B_j(a)|=\Big|\frac{\xi h^{(j)}(a)}{j!}\Big|\le 2^{j+1}\frac{h(3N)}{h(N)} F\cdot N^{-j}.
\end{equation}
Since $1\le j\le c(\log N)^{c-1}$, Lemma~$\ref{help}$ yields 
\begin{equation}\label{bigq}
\begin{split}
2^{j+1}\frac{h(3N)}{h(N)}\le  2^{j+1} e^{c2^{c-1}(\log3)(\log N)^{c-1}}\le e^{2j+8(\log N)^{c-1}}
\le e^{12(\log N)^{c-1}}\text{.}
\end{split}
\end{equation}
Using \eqref{boundcoefq} and \eqref{bigq}, we obtain that for every $j\in \big[c(\log N)^{c-1}\big]$ we have
\begin{align}
\label{eq:12}
e^{-12(\log N)^{c-1}} F\cdot N^{-j}\le |B_j(a)|\le e^{12(\log N)^{c-1}} F\cdot N^{-j}.
\end{align}

\paragraph{\textbf{Step 6: Estimates for $\Delta_I(a)$}}
In view of~\eqref{eq:11}, it remains to estimate $\Delta_I(a)$ for a suitably chosen subset $I\subseteq [n]$. We fix
\[
I\coloneqq\bigg\{j\in\mathbb{N}: \frac{\log F}{\log (N/P)}<j\le c(\log N)^{c-1}\bigg\}\text{.}
\]
Note that for every $j\in I$ we have that $FN^{-j}< F^{-1}N^{j}P^{-2j}$, since 
\[
FN^{-j}< F^{-1}N^{j}P^{-2j}\iff F< (N/P)^j\iff \log(F)< j\log(N/P),
\]
which is guaranteed for $j\in I$. Thus, for this choice of $I\subseteq\big[c(\log N)^{c-1}\big]\subseteq[n]$, using~\eqref{eq:12}, we may estimate the product in~\eqref{eq:3} as follows
\begin{equation}\label{firstboundherenewq}
\begin{split}
\prod_{j\in I}\big(|B_j(a)|+|B_j(a)|^{-1}P^{-2j}\big)&\le \prod_{j\in I}\big(e^{12(\log N)^{c-1}}FN^{-j}+e^{12(\log N)^{c-1}}F^{-1}N^{j}P^{-2j}\big)
\\
&\le e^{12(\log N)^{c-1}|I|}2^{|I|}\prod_{j\in I}F^{-1}N^{j}P^{-2j}
\\
&\le e^{50(\log N)^{2c-2}}\prod_{j\in I}F^{-1}\Big(2e^{2(\log N)^{\tau}}\Big)^{j}\\
&\le e^{50(\log N)^{2c-2}}F^{-|I|}\Big(2e^{2(\log N)^{\tau}}\Big)^{4(\log N)^{2c-2}}
\\
&\le e^{100(\log N)^{2c-2+\tau}}F^{-|I|}\text{,}
\end{split}
\end{equation}
where we have used the estimate $NP^{-2}\le 2e^{2(\log N)^{\tau}}$,
which holds for sufficiently large integers $N\in\mathbb Z_+$.

Next, note that $N/P\ge N^{1/2}$ and, by~\eqref{eq:6}, $F=\xi h(N)\le h(N)^{\frac{c-1}{4}}$, and consequently
\begin{equation*}
\frac{\log F}{\log (N/P)}\le \frac{\frac{c-1}{4}\log(h(N))}{\log (N^{1/2})}\le \frac{c-1}{2}\cdot\frac{(\log N)^{c}}{\log N}<(c-1)(\log N)^{c-1},
\end{equation*}
and thus for sufficiently large integers $N\in\mathbb Z_+$,  we obtain
\[
|I|\ge c\log(N)^{c-1}-(c-1)\log(N)^{c-1}-1=\log(N)^{c-1}-1\ge \frac{1}{2}\log(N)^{c-1}.
\] 
Thus, since $F\ge N^{\rho}\ge 1$ by \eqref{eq:6}, we have that 
\begin{equation}\label{Ffinalnewq}
F^{-|I|}\le N^{-\rho|I|}\le N^{-\frac{\rho}{2}(\log N)^{c-1} }=e^{-\frac{\rho}{2}(\log N)^{c-1}(\log N) }=e^{-\frac{\rho}{2}(\log N)^{c} }.
\end{equation}
Returning to \eqref{firstboundherenewq} and invoking \eqref{Ffinalnewq}, we see that
\begin{equation*}
\prod_{j\in I}\big(|B_j(a)|+|B_j(a)|^{-1}P^{-2j}\big)\le e^{100(\log N)^{2c-2+\tau}}e^{-\frac{\rho}{2}(\log N)^c}
\le e^{-\frac{\rho}{4}(\log N)^c},
\end{equation*}
where the last estimate holds for sufficiently large $N\in\mathbb{Z}_+$ since $2c-2+\tau=2c-2+\left(2-c-\frac{\delta}{2}\right)=c-\frac{\delta}{2}$. Thus, by~\eqref{eq:3}, for sufficiently large integers $N\in\mathbb Z_+$ and every integer $a\in(N,2N]$, we conclude that \begin{equation}\label{finalestanalysisq}
\Delta_I(a)^{\frac{1}{4k^2}}\le(10n\log N)^{\frac{n}{4k^2}}\cdot e^{-\frac{\rho}{16k^2}(\log N)^{c} }.
\end{equation}
The first factor in~\eqref{finalestanalysisq} satisfies $(10n\log N)^{\frac{n}{4k^2}}\le e$, since $n\ge 5(\log N)^{2c-2+\delta}$ and the following estimate 
\begin{equation*}
\frac{n}{4k^2}\log(10n\log N )\le\frac{1}{n^3}\log(10n\log N)\le \frac{1}{n^3}\log\Big(10n^{1+\frac{1}{2c-2+\delta}}\Big) \le 1,
\end{equation*}
is clearly satisfied for sufficiently large $N\in\mathbb Z_+$.

Finally, for the second factor in \eqref{finalestanalysisq} we begin by noting that
\[
k= 2mn\le  200(\log N)^{4c-4+2\delta}(\log\log N),
\]
and taking $N\in\mathbb Z_+$ sufficiently large so that $\frac{\rho(\log N)^{(8-7c)/4}}{64\cdot 10^4(\log\log N)^2}\ge 1$, we may consequently write
\begin{align*}
\frac{\rho}{16k^2}(\log N)^{c}\ge& \frac{\rho (\log N)^{c}}{16\cdot 200^2(\log N)^{8c-8+4\delta}(\log\log N)^2}
\\
=&\frac{\rho}{64\cdot 10^4(\log\log N)^2}(\log N)^{8-7c-4\delta}\ge (\log N)^{(8-7c)/2}.
\end{align*}
Hence, for sufficiently large integers $N\in\mathbb Z_+$ and every integer $a\in(N,2N]$, it follows that
\begin{equation}
\Delta_I(a)^{\frac{1}{4k^2}}\le e\cdot e^{-(\log N)^\frac{8-7c}{2}},
\end{equation}
which combined with \eqref{eq:11} yields
\begin{align*}
J_{k,n}(P)^{\frac{1}{2k^2}}\Delta_I(a)^{\frac{1}{4k^2}} k^{\frac{n}{2k^2}}P^{\frac{n(n+1)}{4k^2}-\frac{1}{k}}
\le e^4 e^{-(\log N)^\frac{8-7c}{2}}.
\end{align*}
Let $\chi\coloneqq\min\big\{\frac{8-7c}{2},\tau\big\}>0$. Combining~\eqref{eq:9} with the preceding bound and~\eqref{eq:4}, we conclude that
\begin{align*}
\Big|\frac{1}{N}\sum_{n=N+1}^{2N}e(h(n)\xi)\Big|\le e^4 e^{-(\log N)^\frac{8-7c}{2}}+4\pi e^{-(\log N)^{\tau}}\le e^5e^{-(\log N)^{\chi}},
\end{align*}
for sufficiently large integers $N\in\mathbb Z_+$,
and the proof of Proposition~\ref{expsumestfull} is complete.
\end{proof}

We close this section by establishing a full  version of the exponential sum estimate from Proposition~\ref{expsumestfull}.

\begin{proposition}\label{expsumestfullfullN}
Let $c\in(1,8/7)$, $\rho'\in(0, \infty)$ and $\theta'\in(0,c-1)$. Then there exist  $\chi'=\chi'(c)\in\mathbb R_+$ and $C=C(c,\theta')\in\mathbb R_+$ such that for every $N\in\mathbb{Z}_+$ and $\xi\in\mathbb{R}$ with $h(N)^{-1}N^{\rho'}\le|\xi|\le h(N)^{\theta'}$ we have
\begin{align}
\label{eq:13}
\Big|\sum_{n\in[N]}e\big(h(n)\xi\big)\Big|\le CNe^{-(\log N)^{\chi'}}.
\end{align}
\end{proposition}
\begin{proof}
We fix $c\in(1,8/7)$, $\rho'\in(0, \infty)$ and $\theta'\in(0,c-1)$, and allow all implicit constants to depend on these parameters. It clearly suffices to prove~\eqref{eq:13} for sufficiently large integers $N\in\mathbb Z_+$. Let $K\coloneqq \big\lfloor \frac{\rho'}{16}(\log N)^{2-c}\big\rfloor$, and estimate as follows
\begin{equation}\label{splitcalc}
\begin{split}
\Big|\sum_{n\in[N]}e\big(h(n)\xi\big)\Big|\le& \sum_{j=0}^{K-1}\Big|\sum_{N/2^{j+1}<n\le N/2^j}e\big(h(n)\xi\big)\Big|+\Big|\sum_{1\le n\le N/2^K
}e\big(h(n)\xi\big)\Big|
\\
\le &\sum_{j=0}^{K-1}\Big|\sum_{N/2^{j+1}<n\le N/2^j}e\big(h(n)\xi\big)\Big|+N2^{-K}.
\end{split}
\end{equation}
We wish to apply Proposition~\ref{expsumestfull} to each dyadic piece above with $\rho=\rho'/2$ and $\theta=\frac{\theta'+(c-1)}{2}\in(\theta',c-1)$. We claim that, for sufficiently large integers $N\in\mathbb Z_+$ and every $j\in\mathbb N_{<K}$, we have
\begin{equation}\label{toapplyexp}
h(N/2^{j+1})^{-1}(N/2^{j+1})^{\rho}\le h(N)^{-1}N^{\rho'},\qquad\text{ and }\qquad h(N)^{\theta'}\le h(N/2^{j+1})^{\theta}.
\end{equation}

To prove the first inequality in~\eqref{toapplyexp}, it suffices to show that $h(N/2^{j+1})^{-1}N^{\rho}\le h(N)^{-1}N^{\rho'}$, which is equivalent to
\[
\frac{h(N)}{h(N/2^{j+1})}\le N^{\rho'/2}.
\]
Note that $N/{2^{j+1}}\ge 2^{j+1}$, since, for sufficiently large integers $N\in\mathbb Z_+$, we have $(2j+2)\log 2\le 2K\log 2\le \log N$, which implies that $N\ge 2^{2j+2}$ as desired. Thus Lemma~$\ref{help}$ is applicable and yields
\begin{align*}
\frac{h(N)}{h(N/2^{j+1})}\le&\exp\Big(c2^{c-1}\log(2^{j+1})(\log(N/2^{j+1}))^{c-1} \Big)
\\
\le& \exp\big(8(j+1)(\log N)^{c-1} \big)
\le\exp\big(8K(\log N)^{c-1} \big)\le\exp\Big(\frac{\rho'}{2}\log N \Big)=N^{\rho'/2}\text{,}
\end{align*}
as desired. To prove the second inequality in~\eqref{toapplyexp}, we note that
\begin{equation*}
h(N)^{\theta'}\le h(N/2^{j+1})^{\theta}\iff \bigg(\frac{h(N)}{h(N/2^{j+1})}\bigg)^{\theta'}\le h(N/2^{j+1})^{\theta-\theta'},
\end{equation*} 
and by the previous estimate we have that 
$\frac{h(N)}{h(N/2^{j+1})}\le N^{\rho}$ for every $j\in\mathbb N_{<K}$, so to establish the estimate it suffices to prove that 
\[
N^{\rho\theta'}\le h\big(N/2^{K}\big)^{\theta-\theta'}.
\]
Taking logarithms on both sides, we see that the last inequality is equivalent to
\begin{equation*}
\rho\theta'\log N\le  (\theta-\theta')\log\big(h\big(N/2^{K}\big)\big)=(\theta-\theta')\big(\log N-K\log2\big)^c,
\end{equation*}
which is clearly true for sufficiently large integers $N\in\mathbb Z_+$ by the definition of $K$.
We have shown that both estimates in~\eqref{toapplyexp} hold for sufficiently large integers $N\in\mathbb Z_+$. Therefore, for every frequency $\xi$ satisfying $h(N)^{-1}N^{\rho'}\le |\xi|\le h(N)^{\theta'}$ and every $j\in\mathbb N_{<K}$, the following inequality holds
\[
h(N/2^{j+1})^{-1}(N/2^{j+1})^{\rho}\le h(N)^{-1}N^{\rho'}\le \xi\le h(N)^{\theta'}\le h(N/2^{j+1})^{\theta}.
\]
We can apply Proposition~\ref{expsumestfull} to conclude that there exists a constant $\chi\in(0,1)$ such that
\[
\Big|\sum_{N/2^{j+1}<n\le N/2^j}e\big(h(n)\xi\big)\Big|\lesssim N/2^{j+1} e^{-(\log (N/2^{j+1}))^\chi}\lesssim N e^{-(\log N)^\chi}.
\]
Taking $\chi'\coloneqq\min\big\{\frac{2-c}{2},\frac{\chi}{2}\big\}>0$ and returning to~\eqref{splitcalc}, we obtain
\begin{align*}
\Big|\sum_{n\in [N]}e\big(h(n)\xi\big)\Big|
\le KN e^{-(\log N)^\chi}+N2^{-K}
\lesssim N e^{-(\log N)^{\chi'}},
\end{align*}
for sufficiently large integers $N\in\mathbb Z_+$. This completes the proof of Proposition~\ref{expsumestfullfullN}.
\end{proof}

\section{One-frequency circle method: Proof of Proposition~\ref{expsumestfullforav}}\label{CircleMethod}
The purpose of this section is to prove
Proposition~\ref{expsumestfullforav}, which will serve as the main input for establishing the
quantitative version of the pointwise convergence result of
Theorem~\ref{thm:mainq}. The analysis is naturally divided
into major and minor arcs, and is completed in
the following two lemmas, which clearly imply Proposition~$\ref{expsumestfullforav}$. We remind the reader that $\varphi\coloneq h^{-1}$ is the compositional inverse of $h(x)=\exp\big((\log x)^c\big)$.
\begin{lemma}[Major arc estimates]\label{MAA}
For every $c\in (1,\infty)$ there exists a constant $C=C(c)\in\mathbb{R}_+$ such that for every $N\in\mathbb{Z}_+$ and $\xi\in\mathbb{T}$ with $\|\xi\|\le h(N)^{-1}N^{\frac{1}{2}}$ we have
\begin{equation}\label{majarcestimate}
\Big|\sum_{n\in [N]}e(\xi\lfloor h(n)\rfloor)-\sum_{2\le m\le h(N)}\varphi'(m)e(\xi m)\Big|\le CN^{\frac{1}{2}}.
\end{equation}
\end{lemma}
\begin{lemma}[Minor arc estimates]\label{mAA}
For every $c\in (1,8/7)$, there exist  constants $C=C(c)\in\mathbb R_+$ and $\chi=\chi(c)\in\mathbb R_+$ such that for every $N\in\mathbb{Z}_+$ and $\xi\in\mathbb{T}$ with $\|\xi\|>h(N)^{-1}N^{\frac{1}{2}}$ we have
\begin{equation}\label{mAest}
\Big|\sum_{n\in [N]}e(\xi\lfloor h(n)\rfloor)\Big|\le CNe^{-(\log N)^{\chi}},\quad\text{ and }\quad \Big|\sum_{2\le m\le h(N)}\varphi'(m)e(\xi m)\Big|\le CNe^{-(\log N)^{\chi}}\text{.}
\end{equation}
\end{lemma}
In the following two subsections we provide the proofs of Lemma~\ref{MAA} and Lemma~\ref{mAA}.

\subsection{Major arc estimates}\label{MajArcSubsection}
This short subsection is devoted to the major arc analysis in Lemma~\ref{MAA}. The small size of the major arc allows us to obtain the power-saving estimate in~\eqref{majarcestimate} in a straightforward manner. Before proceeding with the proof, we make a few brief remarks on the admissible size of the major arc, complementing the discussion in the introduction.

\begin{remark}
\label{rem:1}
The proof of inequality~\eqref{majarcestimate} actually shows that, for every $\xi\in[-1/2,1/2)$, we have
\begin{equation}\label{majgen}
\Big|\sum_{n\in [N]}e(\xi\lfloor h(n)\rfloor)-\sum_{ 2\le m\le h(N)}\varphi'(m)e(\xi m)\Big|\lesssim |\xi|h(N)+1.
\end{equation}
\begin{enumerate}[label*={(\arabic*)}, itemsep=2pt]
\item Although the argument is elementary and, in fact, works for every phase function with some degree of smoothness, an approach yielding better dependence on $|\xi|$ and $h(N)$ does not seem to be available.
Taking~\eqref{majgen} into account, to obtain an estimate with a saving factor $\beta(N)$, that is, a bound of the form $\lesssim \beta(N)N$, one must restrict to frequencies satisfying
$|\xi|\le \beta(N)Nh(N)^{-1}$.

\item This highlights that the threshold $|\xi|\le h(N)^{-(1-\varepsilon)}$ for any $\varepsilon>0$, established in Theorem~2 of~\cite{ApKar} (see also Proposition~\ref{prop:BP} with $\alpha=1-\varepsilon$), is inadequate in the present setting, since making the right-hand side arbitrarily small requires
\begin{equation}\label{restrictiononbeta}
h(N)^{-(1-\varepsilon)}\le Nh(N)^{-1}\beta(N)\iff \frac{h(N)^{\varepsilon}}{N}\le \beta(N)\text{,}
\end{equation}
forcing $\beta(N)$ to cease being a saving factor.

\item In contrast to phase functions with superpolynomial growth, such as $h(N)$, let us note that restrictions of the form~\eqref{restrictiononbeta} are, in fact, acceptable for phase functions of polynomial growth. Indeed, one may choose $\varepsilon>0$ sufficiently small to ensure that $\beta(N)$ is a polynomial saving.

\item Finally, since the major arc frequencies must satisfy $|\xi|\le\beta(N)Nh(N)^{-1}$ for some saving factor $\beta(N)$, the minor arc analysis must accommodate frequencies as small as $Nh(N)^{-1}$, independently of the saving factor one seeks to obtain. Theorem~1 in \cite{Kar} cannot treat frequencies this small. More precisely, the aforementioned theorem requires control of the derivatives of $f(x)\coloneqq\xi h(x)$ of the following form
\begin{equation}\label{karconditionfail}
N^{-\tau s}\le \bigg|\frac{f^{(s)}(x)}{s!}\bigg|\le N^{-\tau's}\text{,}\quad x\in[N,2N] \quad\text{for some $\tau<1$.}
\end{equation}
For an effective application of the theorem to these phases, one has to use this condition for an unbounded number of derivatives in $[c(\log N)^{c-1}]$ as $N\to\infty$. However, for $\xi=Nh(N)^{-1}$, an argument identical to that leading to \eqref{boundcoefq} shows that every $s\in[c(\log N)^{c-1}]$ and $x\in[N,2N]$ satisfy
\[
\bigg|\frac{\xi h^{(s)}(x)}{s!}\bigg|\le e^{12(\log N)^{c-1}}N^{-s+1}\text{,}
\]
and thus for any fixed $\tau\in(0,1)$ making  \eqref{karconditionfail} true, we obtain
\[
N^{-\tau s}\le e^{12(\log N)^{c-1}}N^{-s+1}\Longrightarrow -\tau s\le   12(\log N)^{c-2}-s+1\Longrightarrow s\le \frac{12(\log N)^{c-2}+1}{1-\tau}\text{.}
\]
Therefore the set of derivatives satisfying \eqref{karconditionfail} is bounded, and thus a direct application of the theorem becomes ineffective. In fact, with an identical argument one may show that the theorem becomes ineffective even for frequencies as large as $N^{\rho}h(N)^{-1}$ for every fixed $\rho\ge 1$.
\end{enumerate}
\end{remark}
\begin{proof}[Proof of Lemma~$\ref{MAA}$]\label{majsubsection}
All the implicit constants below may depend on $c\in(1, \infty)$. By the standard identification $\mathbb{T}\equiv[-1/2,1/2)$ and by conjugating, we may assume without loss of generality that the frequency $\xi\in\mathbb T$ satisfies $0\le\xi\le h(N)^{-1}N^{\frac{1}{2}}$.
Finally, it is clear that it suffices to establish \eqref{majarcestimate} for $N\ge 2 $, ensuring that the second sum is nonempty, since $h(x)\ge x$.
\smallskip
\paragraph{\textbf{Step 1: Removing the floor function}} Since $N\le h(N)$, we note that
\begin{equation}
\begin{split}
\label{MA1}
\Big|\sum_{n\in [N]}e(\xi\lfloor h(n)\rfloor)-\sum_{n\in [N]}e(\xi h(n))\Big|
\lesssim  \sum_{n\in [N]}\xi\big(h(n)-\lfloor h(n)\rfloor\big)\le \xi N\le h(N)^{-1}N^{\frac{3}{2}}\le N^{\frac{1}{2}}.
\end{split}
\end{equation}

\paragraph{\textbf{Step 2: Approximating the sum by an integral}} We have
\begin{equation}\label{MA2}
\begin{split}
\Big|\sum_{n\in [N]}e(\xi h(n))-\int_{1}^{N}e(\xi h(t))dt\Big|
&\le1+\sum_{n\in [N-1]}\int_{n}^{n+1}\big|e(\xi h(n))-e(\xi h(t))\big|dt\\
&\lesssim1+\xi\sum_{n\in [N-1]}(h(n+1)-h(n))\\
&\le1+\xi h(N)\lesssim N^{\frac{1}{2}},
\end{split}
\end{equation}
since $\big|e(\xi h(n))-e(\xi h(t))\big|\lesssim \xi| h(n)-h(t)|\lesssim \xi (h(n+1)-h(n))$ for any $t\in [n, n+1]$.

\smallskip
\paragraph{\textbf{Step 3: Approximating the integral by a weighted
sum}} Observe that $\sup_{x\in[2,\infty)}|\varphi'(x)|<\infty$ and
that $\varphi'$ is positive and decreasing. Thus, for
$N\ge 2$, we may estimate as follows
\begin{align}
\label{MA3}
\begin{split}
\Big|\int_{1}^Ne(\xi h(t))dt-\sum_{2\le n\le h(N)}\varphi'(n)e(\xi n)\Big|\le\Big|\int_{2}^{h(N)}e(\xi u)\varphi'(u)du-\sum_{3\le n\le h(N)}\varphi'(n)e(\xi n)\Big|+O(1)\\
\le\sum_{n=3}^{\lfloor h(N)\rfloor}\int_{n-1}^n\big|\varphi'(u)e(u\xi)-\varphi'(u)e(n\xi)\big|du+\sum_{n=3}^{\lfloor h(N)\rfloor}\int_{n-1}^n\big|\varphi'(u)e(n\xi)-\varphi'(n)e(n\xi)\big|du+O(1)
\\
\lesssim\sum_{n=3}^{\lfloor h(N)\rfloor}\int_{n-1}^n\varphi'(u)|n\xi-u\xi|du+\sum_{n=3}^{\lfloor h(N)\rfloor}\int_{n-1}^n \big(\varphi'(n-1)-\varphi'(n)\big)du+O(1)
\\
\le\xi\int_2^{h(N)}\varphi'(u)du +\varphi'(2)-\varphi'(\lfloor h(N)\rfloor )+O(1)\le \xi N+ O(1)\lesssim N^{1/2}\text{,}
\end{split}
\end{align}
where the last estimate follows as in~\eqref{MA1}.
Combining \eqref{MA1}, \eqref{MA2}, \eqref{MA3} yields the desired result.
\end{proof}
\subsection{Minor arc estimates}\label{minarcsubsection}
We are now ready to establish the minor arc estimates~\eqref{mAest} from Lemma~\ref{mAA}. The main tool will be Proposition~\ref{expsumestfullfullN}, applied with $\rho'=1/2$.
\begin{proof}[Proof of Lemma~$\ref{mAA}$]\label{pfl43}
By the standard identification $\mathbb{T}\equiv[-1/2,1/2)$ and by conjugating, we may assume without loss of generality that the frequency $\xi\in\mathbb T$ satisfies $h(N)^{-1}N^{\frac{1}{2}}<\xi\le 1/2$.  
It suffices to establish~\eqref{mAest} for sufficiently large integers $N\in\mathbb Z_+$. The argument proceeds in several steps.
\smallskip
\paragraph{\textbf{Step 1: Fourier expansion of the fractional part}}
We begin with the first estimate in \eqref{mAest}. For every integer $M\ge 2$, $x\in[-1/2,1/2)\setminus\{0\}$ and $y\in\mathbb{R}$ we have that
\[
e(-x\{y\})=\sum_{|m|\le M}\frac{1-e(-x)}{2\pi i (m+x)}e(my)+O\bigg(\min\bigg\{1,\frac{1}{M\|y\|}\bigg\}\bigg)\text{,}
\]
where the implied constant is absolute; see \cite[p.~29]{LPE}. Applying this estimate with $M=\big\lfloor h(N)^{\frac{c-1}{4}}-\frac{1}{2}\big\rfloor$, we decompose the sum below into a main term and an error term
\begin{equation}\label{1stapproxnew}
\sum_{n\in [N]}e(\lfloor h(n)\rfloor\xi)=\sum_{n\in [N]}e(h(n)\xi)\sum_{|m|\le M}\frac{1-e(-\xi)}{2\pi i (m+\xi)}e(mh(n))+\sum_{n\in [N]}e(h(n)\xi)g_M(\xi,n),
\end{equation}
where $|g_M(\xi,n)|\lesssim \min\big\{1,\frac{1}{M\|h(n)\|}\big\}$.
\smallskip
\paragraph{\textbf{Step 2: The error term estimate}} We first bound the second term in~\eqref{1stapproxnew}. Note that
\begin{equation}\label{ESTerrmin}
\Big|\sum_{n\in [N]}e(h(n)\xi)g_M(\xi,n)\Big|\lesssim \sum_{n\in [N]} \min\bigg\{1,\frac{1}{M\|h(n)\|}\bigg\}.
\end{equation}
We now use another  standard Fourier expansion; see~\cite[p.~23]{LPE}, for every $M\in\mathbb Z_+$ and $x\in\mathbb R$, we have
\begin{equation}\label{minexp}
\min\bigg\{1,\frac{1}{M\|x\|}\bigg\}=\sum_{m\in\mathbb{Z}}a_me(mx), \quad\text{ where } \quad  |a_m|\lesssim \min\bigg\{\frac{\log M}{M},\frac{M}{|m|^2}\bigg\}.
\end{equation}
We suppress the dependence of the coefficients $a_m$ on $M$. Using the expansion~\eqref{minexp}, we write\begin{equation}\label{easyboundfornoncanc}
\sum_{n\in [N]}\min\left\{1,\frac{1}{M\|h(n)\|}\right\}=\sum_{n\in [N]}\sum_{m\in\mathbb{Z}}a_me(mh(n))\le \sum_{m\in\mathbb{Z}}|a_m|\Big|\sum_{n\in [N]}e(mh(n))\Big|.
\end{equation}
By Proposition~\ref{expsumestfullfullN} with $\theta'=\frac{c-1}{2}\in(0,c-1)$, there exists  $\chi'=\chi'(c)\in(0, 1)$, such that for every $t\in\mathbb{R}$ with $h(N)^{-1}N^{\frac{1}{2}}\le |t|\le h(N)^{\frac{c-1}{2}}$, we have
\begin{equation}\label{applicationoffullexp}
\Big|\sum_{n\in [N]}e(th(n))\Big|\lesssim Ne^{-(\log N)^{\chi'}}.
\end{equation}
We introduce another truncation parameter $T\coloneqq Me^{(\log N)^{\chi'/2}}$, and note that the estimate~\eqref{applicationoffullexp} is applicable for every $t\in\mathbb Z$ with $1\le |t|\le T$, since $h(N)^{-1}N^{1/2}\le N^{-1/2}\le1$, as $h(N)\ge N$, while, for sufficiently large integers $N\in\mathbb Z_+$, we have  $T\le h(N)^{\frac{c-1}{4}}e^{(\log N)^{\chi'/2}}\le h(N)^{\frac{c-1}{2}}$, where the first inequality follows immediately from the definition of $M$ and the second from the fact that $\chi'/2<c$ and 
\[
e^{(\log N)^{\chi'/2}}\le h(N)^{\frac{c-1}{4}}\iff (\log N)^{\chi'/2}\le \frac{c-1}{4}\log(h(N))=\frac{c-1}{4}(\log N)^c.
\]
Combining~\eqref{applicationoffullexp} with~\eqref{easyboundfornoncanc}, the bounds for $|a_m|$ in~\eqref{minexp}, and the fact that $\chi'/2<c$, we obtain, for sufficiently large integers $N\in\mathbb Z_+$, that
\begin{equation}\label{easyboundfinal}
\begin{split}
\sum_{n\in [N]}\min\left\{1,\frac{1}{M\|h(n)\|}\right\}\le&
\sum_{m\in\mathbb{Z}}|a_m|\Big|\sum_{n\in [N]}e(mh(n))\Big|\\
=&|a_0|N+\sum_{1\le |m|\le T}|a_m|\Big|\sum_{n\in [N]}e(mh(n))\Big|+\sum_{|m|> T}|a_m|N
\\
\lesssim&\frac{N\log M}{M}+\sum_{1\leq|m|\leq T}\frac{\log M}{M}Ne^{-(\log N)^{\chi'}}+\sum_{|m|>T}\frac{M}{|m|^2}N
\\
\lesssim&N\cdot\frac{\log M}{M}+N\cdot\frac{Te^{-(\log N)^{\chi'}}\log M}{M}+N\cdot\frac{M}{T}
\lesssim Ne^{-(\log N)^{\chi'/2}},
\end{split}
\end{equation}
since  $\log M\le\log (h(N))\le  (\log N)^2$ and also that $M\ge\frac{1}{2} h(N)^{\frac{c-1}{4}}$.
Combining \eqref{ESTerrmin} with \eqref{easyboundfinal} yields
\begin{align}
\label{eq:14}
\Big|\sum_{n\in [N]}e(h(n)\xi)g_M(\xi,n)\Big|\lesssim Ne^{-(\log N)^{\chi'/2}},
\end{align}
and this completes the estimate of the second summand in~\eqref{1stapproxnew}.
\smallskip
\paragraph{\textbf{Step 3: The main term estimate}}
We now bound the first term in~\eqref{1stapproxnew}. Using
\eqref{applicationoffullexp}  note that
\begin{equation}
\label{eq:15}
\begin{split}
\Big|\sum_{n\in [N]}e(h(n)\xi)&\sum_{|m|\le M}\frac{1-e(-\xi)}{2\pi i (m+\xi)}e(mh(n))\Big|\le \sum_{|m|\le M}\frac{|1-e(-\xi)|}{2\pi |m+\xi|}\Big|\sum_{n\in [N]}e((\xi+m)h(n))\Big|
\\
\lesssim&\Big(\sum_{1\le |m|\le M}\frac{1}{|m|}\Big)\sup_{1\le |m|\le M}\Big|\sum_{n\in [N]}e((\xi+m)h(n))\Big|
+\Big|\sum_{n\in [N]}e(\xi h(n))\Big|
\\
\lesssim& \log M\sup_{h(N)^{-1}N^{1/2}\le |t|\le M+1/2}\Big|\sum_{n\in [N]}e(th(n))\Big|
\\
\le& (\log N)^2\sup_{h(N)^{-1}N^{1/2}\le |t|\le h(N)^{(c-1)/4}}\Big|\sum_{n\in [N]}e(th(n))\Big|
\\
\le& (\log N)^2Ne^{-(\log N)^{\chi'}}\lesssim Ne^{-(\log N)^{\chi'/2}}.
\end{split}
\end{equation}
Combining~\eqref{eq:15}, \eqref{eq:14}, and~\eqref{1stapproxnew}, the first estimate in~\eqref{mAest} follows.
\smallskip
\paragraph{\textbf{Step 4: Estimates for dyadic weighted exponential sums}}
We now turn to the second estimate in~\eqref{mAest}. Assume that $P\in\mathbb Z_+$ is sufficiently large and let $S_{n}(\xi)\coloneqq\sum_{m=\lfloor h(P)\rfloor+1}^ne(m\xi)$ for $n\ge \lfloor h(P)\rfloor+1$.  Then, $|S_{n}(\xi)|\lesssim |\xi|^{-1}$ for any $n\in\mathbb Z_+$, and by summation by parts \eqref{eq:31}, we obtain
\begin{equation}\label{sbpdyad}
\begin{split}
\Big|\sum_{h(P)<n\le h(eP)}\varphi'(n)e(n\xi)\Big|&=
\Big|S_{\lfloor h(eP)\rfloor}(\xi)\varphi'(\lfloor h(eP)\rfloor)-\sum_{n=\lfloor h(P)\rfloor+1}^{\lfloor h(eP)\rfloor -1}S_{n}(\xi)\big(\varphi'(n+1)-\varphi'(n)\big)\Big|
\\
&\lesssim |\xi|^{-1}\varphi'(\lfloor h(eP)\rfloor)+|\xi|^{-1}\sum_{n=\lfloor h(P)\rfloor+1}^{\lfloor h(eP)\rfloor-1}\big(\varphi'(n)-\varphi'(n+1)\big)\\
&\lesssim |\xi|^{-1}\varphi'( h(P))\le N^{-\frac{1}{2}}\frac{Ph(N)}{h(P)},
\end{split}
\end{equation}
where the estimate from the first to the second line follows from the fact that $\varphi'$ is eventually decreasing, while the last estimate follows from the identity $\varphi'(h(P))=\frac{1}{h'(P)}=\frac{P}{h(P)c(\log P)^{c-1}}$ and $\xi>N^{1/2}h(N)^{-1}$.
\smallskip
\paragraph{\textbf{Step 5: Estimates for full weighted exponential sums}}
Let $K\coloneqq \big\lfloor \frac{1}{16}(\log N)^{2-c}\big\rfloor$, and note that by the estimate \eqref{sbpdyad} we obtain that
\begin{align}\label{splitcalcold}
\begin{split}
\Big|\sum_{2\le n\le h(N)}\varphi'(n)e(n\xi)\Big|&\le \sum_{j=0}^{K-1}\Big|\sum_{h(N/e^{j+1})<n\le h(N/e^j)}\varphi'(n)e(n\xi)\Big|+\Big|\sum_{2\le n\le h(N/e^K)
}\varphi'(n)e(n\xi)\Big|
\\
&\lesssim N^{\frac{1}{2}}\sum_{j=0}^{K-1}\frac{h(N)}{h(N/e^{j+1})}+N/e^K+O(1)\text{,}
\end{split}
\end{align}
where the last estimate for the second summand follows from~\eqref{MA3} with $\xi=0$.

Note that, for sufficiently large integers $N\in\mathbb Z_+$, we have $K\ge \frac{1}{32}(\log N)^{2-c}$ and $\frac{N}{e^{j+1}}\ge e^{j+1}$ for every $j\in\mathbb N_{<K}$, since $2j+2\le 2K\le \frac{1}{8}(\log N)^{2-c}\le \log N$. Therefore, by Lemma~\ref{help}, we obtain
\[
\frac{h(N)}{h(N/e^{j+1})}\le\exp\Big(c2^{c-1}\log(e^{j+1})\big(\log \big(N/e^{j+1}\big)\big)^{c-1}\Big)\le \exp\Big(\frac{1}{4}\log N\Big)=N^{1/4}
\]
Returning to \eqref{splitcalcold} and using this bound, we conclude that
\begin{align*}
\Big|\sum_{2\le n\le h(N)}\varphi'(n)e(n\xi)\Big|\lesssim
N^{\frac{3}{4}}(\log N)+Ne^{-\frac{1}{32}(\log N)^{2-c}}\lesssim Ne^{-(\log N)^{\frac{2-c}{2}}}.
\end{align*}
This establishes the second estimate in \eqref{mAest}, and the proof of Lemma~$\ref{mAA}$ is complete.
\end{proof}

\section{Proof of Theorem~\ref{thm:mainq}}\label{sectionOrbit}
In this section we prove Theorem \ref{thm:mainq}. By the Caledr{\'o}n transference principle \cite{C1}, see also \cite{Kosz} the variational inequality from \eqref{eq:29}  follows from its integer shift system counterpart. Namely, it suffices to prove 
that for all $p\in(1, \infty)$ and $r\in(2, \infty)$ there exists  $C=C(h, p)\in\mathbb R_+$ such that for all $f\in \ell^p(\mathbb Z)$, we have 
\begin{equation}
\label{eq:16}
\big\|V^r\big(A_{N;\mathbb Z}^{\lfloor h\rfloor}f:N\in\mathbb{Z}_+\big)\big\|_{\ell^p(\mathbb{Z})}\le C\frac{r}{r-2} \|f\|_{\ell^p(\mathbb{Z})},
\end{equation}
where $A_{N;\mathbb Z}^{\lfloor h\rfloor}$ was defined in \eqref{eq:32}. The key ingredients are Proposition~\ref{expsumestfullforav} and the following lemma, ultimately allowing us to reduce the problem to the $r$-variational estimates for the discrete Hardy--Littlewood averaging operator, established in~\cite[Theorem B]{JKRW}.
\begin{lemma}\label{correctosccomp}
Let $(\lambda_s^k)_{k, s\in\mathbb{Z}_+}$ be non-negative real numbers satisfying the following: 
\begin{itemize}
\item[i)] There exists $\Lambda\in\mathbb R_+$ such that for every $k\in\mathbb Z_+$ we have that $\sum_{s=1}^{\infty}\lambda^k_s=\Lambda$.
\item[ii)] For any fixed $N\in\mathbb{Z}_+$ we have that $\sum_{s=1}^{N}\lambda_{s}^k$ is decreasing in $k$.
\end{itemize}
Then for any $r\in[1,\infty)$ and any sequence $(a_n)_{n\in\mathbb{Z}_+}$ of complex numbers we have
\[
V^r\Big(\sum_{s=1}^{\infty}\lambda_s^ka_s:\,k\in\mathbb{Z}_+\Big)\le \Lambda \cdot V^r(a_n:\,n\in\mathbb{Z}_+).
\] 
\end{lemma}
\begin{proof}
For the proof we refer to  \cite[Lemma~2]{UA}.
\end{proof}

\begin{proof}[Proof of inequality \eqref{eq:16}] We fix $p\in(1, \infty)$ and $r\in(2, \infty)$. Fix $p\in(1,\infty)$ and choose $p_0>1$, sufficiently close to $1$, so that $p\in(p_0,p_0')$. Next, choose $\tau\in(0,1)$ such that $\tau<\frac{1}{2}\min\{p_0-1,1\}$.

The argument will proceed in a few steps.
\smallskip
\paragraph{\textbf{Step 1: Passing to sublacunary sequences}} Following~\cite[Lemma~1.3]{JSW}, we decompose~\eqref{eq:16} into the long and short variations (the first and second terms, respectively) as follows
\begin{align*}
\big\|V^r\big(A_{N;\mathbb Z}^{\lfloor h\rfloor}f:N\in\mathbb{Z}_+\big)\big\|_{\ell^p(\mathbb{Z})}&\lesssim
\big\|V^r\big(A_{\lfloor 2^{k^{\tau}}\rfloor;\mathbb Z}^{\lfloor h\rfloor}f:k\in\mathbb{N}\big)\big\|_{\ell^p(\mathbb{Z})}\\
&+\Bigg\|\bigg(\sum_{k=0}^{\infty}V^2\Big(A_{N;\mathbb Z}^{\lfloor h\rfloor}f:N\in\big[\big\lfloor 2^{k^{\tau}}\big\rfloor,\big\lfloor 2^{(k+1)^{\tau}}\big\rfloor\big)\Big)^2\bigg)^{1/2}\Bigg\|_{\ell^p(\mathbb{Z})}.
\end{align*}
The advantage of this decomposition is that the short variations are taken over intervals of sublacunary length. Combined with the estimate
\begin{align}
\label{eq:17}
\big\|A_{N+1;\mathbb Z}^{\lfloor h\rfloor}f-A_{N;\mathbb Z}^{\lfloor h\rfloor}f\big\|_{\ell^p(\mathbb{Z})}
\lesssim N^{-1}\|f\|_{\ell^p(\mathbb{Z})},
\end{align}
this implies, by the argument used in~\cite[Estimate~(3.19)]{MSZ3}, that
\[
\Bigg\|\bigg(\sum_{k=0}^{\infty}V^2\Big(A_{N;\mathbb Z}^{\lfloor h\rfloor}f:N\in\big[\big\lfloor 2^{k^{\tau}}\big\rfloor,\big\lfloor 2^{(k+1)^{\tau}}\big\rfloor\big)\Big)^2\bigg)^{1/2}\Bigg\|_{\ell^p(\mathbb{Z})}
\lesssim \|f\|_{\ell^p(\mathbb{Z})}.
\]
\paragraph{\textbf{Step 2: Passing to weighted averages}} Let $\varphi:[1,\infty)\to(0,\infty)$ denote the compositional inverse of $h$, given explicitly by
$\varphi(x)\coloneqq \exp\big((\log x)^{1/c}\big)$. Observe that $\varphi'$ is defined on $(1,\infty)$ but not at $x=1$. For the purposes of our argument, we adopt the convention $\varphi'(1)\coloneqq \varphi'(2)$. Then, for every $x\in\mathbb Z$, $N\in\mathbb Z_+$, and any finitely supported function $f\colon\mathbb Z\to\mathbb C$, we define the weighted averages by setting
\begin{align*}
A_{N; \varphi', \mathbb Z}^{{\rm n}}f(x)\coloneqq \frac{1}{\Phi(N)}\sum_{n\in[h(N)]}\varphi'(n)f(x-n),
\end{align*}
where $\Phi(N)\coloneqq \sum_{n\in[h(N)]}\varphi'(n)$. By \eqref{MA3} with $\xi=0$, it is not difficult to see  that there exists a constant $C\in\mathbb R_+$ such that $|\Phi(N)-N|\le C$ for every $N\in\mathbb Z_+$. Therefore, if $B_N\in\big\{A_{N;\mathbb Z}^{\lfloor h\rfloor}, A_{N; \varphi', \mathbb Z}^{{\rm n}}\big\}$ and $q\in\{p_0, p_0'\}$, then for every $f\in\ell^q(\mathbb Z)$, we have 
\begin{align}
\label{eq:19}
\|B_Nf\|_{\ell^q(\mathbb Z)}\lesssim \|f\|_{\ell^q(\mathbb Z)}.
\end{align}
Now by simple properties of $r$-variations, see \eqref{eq:25}, we obtain that
\begin{align*}
\big\|V^r\big(A_{\lfloor 2^{k^{\tau}}\rfloor;\mathbb Z}^{\lfloor h\rfloor}f:k\in\mathbb{N}\big)\big\|_{\ell^p(\mathbb{Z})}\lesssim
\big\|V^r\big(A_{\lfloor 2^{k^{\tau}}\rfloor; \varphi',\mathbb Z}^{{\rm n}}f:k\in\mathbb{N}\big)\big\|_{\ell^p(\mathbb{Z})}
+\sum_{k=0}^{\infty}\big\|A_{\lfloor 2^{k^{\tau}}\rfloor;\mathbb Z}^{\lfloor h\rfloor}f-A_{\lfloor 2^{k^{\tau}}\rfloor; \varphi', \mathbb Z}^{{\rm n}}f\big\|_{\ell^p(\mathbb{Z})}.
\end{align*}
The last series is summable over $k\in\mathbb N$, since
\begin{align}
\label{eq:20}
\big\|A_{\lfloor 2^{k^{\tau}}\rfloor;\mathbb Z}^{\lfloor h\rfloor}f-A_{\lfloor 2^{k^{\tau}}\rfloor; \varphi', \mathbb Z}^{{\rm n}}f\big\|_{\ell^p(\mathbb{Z})}\lesssim k^{-2}\|f\|_{\ell^p(\mathbb{Z})}.
\end{align}
Indeed, by \eqref{eq:19} for $q\in\{p_0, p_0'\}$ observe that
\begin{align}
\label{eq:21}
\big\|A_{N;\mathbb Z}^{\lfloor h\rfloor}f-A_{N; \varphi', \mathbb Z}^{{\rm n}}f\big\|_{\ell^q(\mathbb{Z})}\lesssim \|f\|_{\ell^q(\mathbb{Z})}.
\end{align}
For $q=2$, there exists $\chi\in(0,1)$ such that the following stronger estimate holds
\begin{align}
\label{eq:22}
\big\|A_{N;\mathbb Z}^{\lfloor h\rfloor}f-A_{N; \varphi', \mathbb Z}^{{\rm n}}f\big\|_{\ell^2(\mathbb{Z})}\lesssim e^{-(\log N)^{\chi}}\|f\|_{\ell^2(\mathbb{Z})}.
\end{align}
By Plancherel's theorem, inequality~\eqref{eq:22} immediately follows from Proposition~\ref{expsumestfullforav} and the fact that $|\Phi(N)-N|=O(1)$. Now, applying~\eqref{eq:21} and~\eqref{eq:22} with $N=\lfloor 2^{k^{\tau}}\rfloor$ and interpolating between these two bounds, we obtain~\eqref{eq:20}. The problem is now reduced to proving the following inequality
\begin{align}
\label{eq:23}
\big\|V^r\big(A_{N; \varphi',\mathbb Z}^{{\rm n}}f:N\in\mathbb{Z}_+\big)\big\|_{\ell^p(\mathbb{Z})}\lesssim_p \frac{r}{r-2} \|f\|_{\ell^p(\mathbb{Z})}.
\end{align}
\paragraph{\textbf{Step 3: Passing to the Hardy--Littlewood averages}} We now prove inequality \eqref{eq:23}. For every $x\in\mathbb Z$, $N\in\mathbb Z_+$, and any finitely supported function $f\colon\mathbb Z\to\mathbb C$, we define the classical  Hardy--Littlewood averages by setting
\begin{align*}
A_{N;\mathbb Z}^{{\rm n}}f(x)\coloneqq \frac{1}{N}\sum_{n\in[N]}f(x-n).
\end{align*}
By summation by parts, for every $k\in\mathbb Z_+$, we have
\begin{equation}\label{difsbp}
\begin{split}
A_{k; \varphi',\mathbb Z}^{{\rm n}}f(x)&=\sum_{s=1}^{\lfloor h(k)\rfloor}\frac{\varphi'(s)}{\Phi(k)}f(x-s)\\
&=A_{\lfloor h(k)\rfloor;\mathbb Z}^{{\rm n}}f(x)\frac{\lfloor h(k)\rfloor\varphi'(\lfloor h(k)\rfloor)}{\Phi(k)}
+\sum_{s=1}^{\lfloor h(k)\rfloor-1}A_{s;\mathbb Z}^{{\rm n}}f(x)\frac{s(\varphi'(s)-\varphi'(s+1))}{\Phi(k)}\\
&=\sum_{s=1}^{\infty}\lambda^k_{s}A_{k;\mathbb Z}^{{\rm n}}f(x),
\end{split}
\end{equation}
where
\[ 
\lambda_{s}^k\coloneqq \left\{
\begin{array}{ll}     
\frac{s(\varphi'(s)-\varphi'(s+1))}{\Phi(k)}
&
     \text{if }1 \le s\le \lfloor h(k)\rfloor-1,\\
      \frac{\lfloor h(k)\rfloor\varphi'(\lfloor h(k)\rfloor)}{\Phi(k)}
&\text{if }s=\lfloor h(k)\rfloor,\\
      0&\text{if }s>\lfloor h(k)\rfloor.\\
\end{array} 
\right. 
\]
Since the sequence $\varphi'(n)$ is positive and nonincreasing, we obtain that $\lambda_{s}^k\ge 0$, and by  summation by parts we may write
\begin{equation}\label{easysbp}
\sum_{s=1}^{\infty}\lambda_s^k=\sum_{s=1}^{\lfloor h(k)\rfloor-1}\frac{s(\varphi'(s)-\varphi'(s+1))}{\Phi(k)}
+\frac{\lfloor h(k)\rfloor\varphi'(\lfloor h(k)\rfloor)}{\Phi(k)}=1\text{.}
\end{equation}
Finally, we note that, for any fixed $N\in\mathbb N$, the sequence $\big(\sum_{s=1}^N\lambda_s^k\big)_{k\in\mathbb Z_+}$ is nonincreasing in $k$, since
\[
\sum_{s=1}^N\lambda_s^k=\left\{
\begin{array}{ll}
\sum_{s=1}^N \frac{s(\varphi'(s)-\varphi'(s+1))}{\Phi(k)}
&\text{if }1\le N\le \lfloor h(k)\rfloor-1,\\
1&\text{if }N\ge \lfloor h(k)\rfloor,\\
\end{array}
\right.
\]
and, for any $1\le N\le \lfloor h(k)\rfloor-1$, the nonnegativity of the summands and~\eqref{easysbp} imply that
\[
\sum_{s=1}^N \frac{s(\varphi'(s)-\varphi'(s+1))}{\Phi(k)}
\le 1.
\]
Thus Lemma~$\ref{correctosccomp}$ is applicable and yields that
\begin{align*}
\big\|V^r\big(A_{N; \varphi', \mathbb Z}^{{\rm n}}f:N\in\mathbb{Z}_+\big)\big\|_{\ell^p(\mathbb{Z})}&=\Big\|V^{r}\Big(\sum_{s=1}^{\infty}\lambda^k_{s}A_{s;\mathbb Z}^{\rm n}f:\,k\in\mathbb{Z}_+\Big)\Big\|_{\ell^p(\mathbb{Z})}\\
&\lesssim \|V^r(A_{N;\mathbb Z}^{\rm n} f:N\in\mathbb{Z}_+)\|_{\ell^p(\mathbb{Z})}\lesssim\frac{r}{r-2} \|f\|_{\ell^p(\mathbb{Z})},
\end{align*}
where the last inequality follows from \cite[Theorem B]{JKRW}. This complete the proof of the theorem.
\end{proof}


\begin{thebibliography}{99}

\bibitem{Bellow2}
\textsc{M.~Akcoglu, A.~Bellow, R.L.~Jones, V.~Losert, K.~Reinhold-Larsson, M.~Wierdl.} \newblock{The strong sweeping out property
for lacunary sequences, Riemann sums, convolution powers, and related
matters.}
\newblock{Ergodic Theory and Dynamical Systems, 16(2),
207–253. (1996). doi:10.1017/S0143385700008798}

\bibitem{frpr} 
\textsc{E.~Bahnson, L.~Daskalakis, A.~Dohadwala, I.~Shah.}
\newblock{Pointwise Ergodic Theorems Along Fractional Powers of Primes.}
\newblock{International Mathematics Research Notices, Volume 2025, Issue 15,
August 2025.}


\bibitem{Bel}
\textsc{A.~Bellow.}
\newblock {Measure Theory
Oberwolfach 1981. Proceedings of the Conference held at Oberwolfach,
June 21-27, 1981.}
\newblock {Lecture Notes in Mathematics {\bf
945}, editors D. K\"olzow and D. Maharam-Stone. Springer-Verlag Berlin
Heidelberg (1982).}
\newblock {Section: Two problems submitted by
A.~Bellow, pp.~429--431.}


\bibitem{Bellow}
\textsc{A.~Bellow.}
\newblock{On ``bad universal'' sequences in
ergodic theory (II).}
\newblock{In: Belley, JM., Dubois, J., Morales, P. (eds)
Measure Theory and its Applications. Lecture Notes in Mathematics, vol
1033. Springer, Berlin,
Heidelberg. https://doi.org/10.1007/BFb0099847, (1983).}

\bibitem{Bir}
\textsc{G.~Birkhoff.}
\newblock{Proof of the ergodic
theorem.}
\newblock{Proc. Natl. Acad. Sci. USA~{\bf 17} (1931), no.~12, pp.~656--660.}

\bibitem{BH}
\textsc{J.R.~Blum, D.L.~Hanson.}
\newblock{On the mean ergodic theorem for subsequences.}
\newblock{Bull. Amer. Math. Soc. {\bf 66} (1960), no. 4, pp.~308--311.}

\bibitem{BKQW}
\textsc{M. Boshernitzan, G. Kolesnik, A. Quas, M. Wierdl.}
\newblock{Ergodic averaging sequences.}
\newblock{J. Anal. Math. {\bf 95} (2005), pp.~63--103.}

\bibitem{Bowie}
\textsc{M. Boshernitzan, M. Wierdl.}
\newblock{Ergodic theorems along sequences and Hardy fields.}
\newblock{Proc. Nat. Acad. Sci. U.S.A. {\bf 93} (1996), no. 16,  pp. 8205--8207.}

\bibitem{B1}
\textsc{J.~Bourgain.}
\newblock{On the maximal ergodic
theorem for certain subsets of the integers.}  \newblock{Israel
J.~Math.~{\bf 61} (1988), pp.~39--72.}
		
\bibitem{B2}
\textsc{J.~Bourgain.}
\newblock{On the pointwise ergodic
theorem on $L^p$ for arithmetic sets.}  \newblock{Israel J.~Math.~{\bf
61} (1988), pp.~73--84.}
		
\bibitem{B3}
\textsc{J.~Bourgain.}
\newblock{Pointwise ergodic
theorems for arithmetic sets. With an appendix by the author,
H.~Furstenberg, Y.~Katznelson, and D.S.~Ornstein.}
\newblock{Inst. Hautes Etudes Sci. Publ. Math.~{\bf 69} (1989),
pp.~5--45.}

\bibitem{BDG} \textsc{J.~Bourgain, C.~Demeter, L.~Guth.}
\newblock{Proof of the main conjecture in Vinogradov's Mean Value
Theorem for degrees higher than three.}  \newblock{Ann. of Math.~{\bf
184} (2016), no.~2, pp.~633--682.}		
    
\bibitem{BM}
\textsc{Z.~Buczolich, R.D.~Mauldin.}  \newblock{Divergent
square averages.}
\newblock{Ann. of Math.~{\bf 171} (2010), no.~3,
pp.~1479--1530.}

\bibitem{ApKar}
\textsc{J.~Brüdern, A.~Perelli.}
\newblock{Goldbach Numbers in Sparse Sequences.}
\newblock{Annales de l’institut Fourier, 48, (1998), 353-378. 
https://doi.org/10.5802/aif.1621.}

\bibitem{C1}
\textsc{A.~Calder\'{o}n.}
\newblock{Ergodic theory and
translation invariant operators.}
\newblock{Proc. Natl. Acad. Sci. USA~{\bf 59} (1968), pp.~349--353.}

\bibitem{C}
\textsc{M.~Christ.}
\newblock{A weak type $(1,1)$
inequality for maximal averages over certain sparse sequences.}
\newblock{Preprint: arXiv:1108.5664.}


\bibitem{WT11LD}
\textsc{L.~Daskalakis.}
\newblock{Weak-type $(1,1)$ inequality for discrete maximal functions and pointwise ergodic theorems along thin arithmetic sets.}
\newblock{J Fourier Anal Appl 30, 37 (2024).}

\bibitem{brpoly}
\textsc{L.~Daskalakis.}
\newblock{Pointwise convergence of ergodic averages along quadratic bracket polynomials.} 
\newblock{Preprint: arXiv:2510.27590.}



\bibitem{Fur3}
\textsc{H.~Furstenberg.}
\newblock{Problems Session, Conference on Ergodic Theory and Applications.}  \newblock{University
of New Hampshire, Durham, NH, June 1982.}

\bibitem{Krengel}
\textsc{U.~Krengel.}
\newblock{On the Individual Ergodic Theorem for Subsequences.}
\newblock{Ann. Math. Stat. {\bf 42} (1971), no.~3, pp.~1091--1095.}
                
\bibitem{LPE}
\textsc{A.~Iosevich, B.~Langowski, M.~Mirek, T.Z.~Szarek.}
\newblock{Lattice points problem, equidistribution and ergodic theorems for certain arithmetic spheres.}
\newblock{Mathematische Annalen 388, (2024), pp.~2041--2120.}

\bibitem{IWKO}
\textsc{H.~Iwaniec, E.~Kowalski.}
\newblock{Analytic Number Theory.}
\newblock{Vol.~53, Amer. Math. Soc. Colloquium Publications, Providence RI, (2004).}	

\bibitem{JKRW}
\textsc{R.L.~Jones, R.~Kaufman, J.~Rosenblatt,
M.~Wierdl.}
\newblock{Oscillation in ergodic theory.}
\newblock{Ergodic Theory Dynam. Systems~{\bf 18} (1998), no.~4,
pp.~889--935.}
		
\bibitem{JSW}
\textsc{R.L.~Jones, A.~Seeger, J.~Wright.}
\newblock{Strong variational and jump inequalities in harmonic
analysis.}
\newblock{Trans. Amer. Math. Soc.~{\bf 360} (2008), no.~12, pp.~6711--6742.}

\bibitem{JW}
\textsc{R.L.~Jones, M.~Wierdl.}
\newblock{Convergence and divergence
of ergodic averages.}
\newblock{Ergodic Theory Dynam. Systems 14.3, pp.~515--
535. (1994). ISSN: 0143-3857. DOI: 10 . 1017 / S0143385700008002. URL:
https://doi.org/10.1017/S0143385700008002.}

\bibitem{Kar}
\textsc{A.A.~Karatsuba.}
\newblock{Estimates for trigonometric sums by Vinogradov’s method, and some applications.} \newblock{Proc. Steklov Inst. Math., 112:251--265, 1971.}

\bibitem{Kosz}
\textsc{D.~Kosz.}
\newblock {Sharp constants in
inequalities admitting the Calder\'{o}n transference principle.}
\newblock {Ergodic Theory Dynam. Systems~{\bf 44} (2024),
pp.~1597--1608.}

\bibitem{LaV1}
\textsc{P.~LaVictoire.}
\newblock{Universally
$L^1$-Bad Arithmetic Sequences.}  \newblock{J.~Anal. Math.~{\bf 113}
(2011), no.~1, pp.~241--263.}

\bibitem{Mes}
\textsc{N.~Mehlhop, W.~S{\l}omian.}
\newblock{Oscillation and jump inequalities for the polynomial ergodic averages along multi-dimensional subsets of primes .}
\newblock{Math. Ann.  {\bf 388} (2024), pp. 2807--2842.}

\bibitem{Mirek}
\textsc{M.~Mirek.}
\newblock{Weak type $(1, 1)$ inequalities for
discrete rough maximal functions.}
\newblock{J. Anal. Mat. 127 (2015), 303--337.}



\bibitem{MOE}
\textsc{M.~Mirek, T.Z.~Szarek, J.~Wright.} \newblock{Oscillation
inequalities in ergodic theory and analysis: one-parameter and
multi-parameter perspectives.}
\newblock{Rev. Mat. Iberoam. 38 (2022), no.~7,
2249--2284.}

\bibitem{UA}
\textsc{M.~Mirek, B.~Trojan, P.~Zorin-Kranich.}
\newblock{Variational estimates for averages and truncated singular integrals along the prime numbers.}
\newblock{Transactions of the American Mathematical Society 369, (2017), no.~8, 5403--5423.}

\bibitem{MSZ3}
\textsc{M.~Mirek, E.M.~Stein, P.~Zorin-Kranich.}
\newblock{Jump inequalities for translation-invariant operators of
Radon type on $\mathbb Z^d$.}
\newblock{Adv. Math.~{\bf 365} (2020), article
no.~107065.}


\bibitem{MW}
\textsc{S.~Mondal, M.~Roy, M.~Wierdl.}
\newblock{Sublacunary sequences that are strong sweeping out.}
\newblock{New York Journal of Mathematics, 29, (2023), 1060--1074.} 

\bibitem{Nair}
\textsc{R.~Nair.}
\newblock{On polynomials in primes and J.~Bourgain's circle method approach to ergodic theorems II.}
\newblock{
Studia Mathematica, {\bf105}, (1993), 207--233.}

\bibitem{Lillian}
\textsc{L.B.~Pierce.}
\newblock{The Vinogradov Mean Value Theorem [after Wooley, and Bourgain, Demeter and Guth].}
\newblock{Ast\'Erisque. 2017, Jul 4.}

\bibitem{RW}
\textsc{J.~Rosenblatt, M.~Wierdl.}
\newblock{Pointwise ergodic theorems via harmonic analysis. In Proc. Conference on Ergodic Theory (Alexandria, Egypt, 1993).}
\newblock{London Mathematical Society Lecture Notes, {\bf 205}, (1995), pp.~3--151.}

\bibitem{Trojan}
\textsc{B.~Trojan.}
\newblock{Variational estimates for discrete operators modeled on multi-dimensional polynomial subsets of primes.}
\newblock{Math. Ann. {\bf374} (2019), pp.~1597--1656.}

\bibitem{UZ}
\textsc{R.~Urban, J.~Zienkiewicz.}
\newblock{Weak type $(1,1)$ estimates for a class of discrete rough maximal functions.}
\newblock{Math. Res. Lett. {\bf14} (2007), no.~2, pp.~227--237.}

\bibitem{V}
\textsc{I.M.~Vinogradov.}
\newblock{The method of trigonometrical sums in the theory of numbers.}
\newblock{Interscience Publishers New York, (1954)}.

\bibitem{W}
\textsc{T.D.~Wooley.}
\newblock{Vinogradov’s mean value theorem via efficient congruencing.}
\newblock{Ann. of Math. (2), 175(3):1575--1627, 2012.}

\bibitem{Wierdl}
\textsc{M.~Wierdl.}
\newblock{Pointwise ergodic theorem along the prime numbers.}
\newblock{Israel J. Math. {\bf64} (1988), no.~3, pp.~315--336.}

\bibitem{Wierdlphd}
\textsc{M.~Wierdl.}
\newblock{Almost everywhere convergence and recurrence along subsequences in ergodic theory.}
\newblock{Ph.D. thesis, Ohio State University, 1989.}

\end{thebibliography}
\end{document}